\documentclass[]{amsart}
\usepackage[utf8]{inputenc}
\usepackage{amsmath}
\usepackage{amssymb}
\usepackage{amsthm}
\usepackage{tikz}
\usepackage{xcolor}
\usepackage{hyperref}
\usepackage{url}
\usepackage{braket}
\usepackage{nccmath}
\usepackage{stmaryrd}
\usepackage{xfrac}
\usetikzlibrary{patterns,snakes}
\usetikzlibrary{arrows}
\usetikzlibrary{decorations.markings}
\DeclareMathAlphabet\mathbfcal{OMS}{cmsy}{b}{n}

\hypersetup{
 colorlinks=true,       
    linkcolor=cyan,          
    citecolor=green,        
    filecolor=magenta,      
   urlcolor=cyan           
   }

\numberwithin{equation}{section}
\numberwithin{figure}{section}

\allowdisplaybreaks
\newtheorem{innercustomgeneric}{\customgenericname}
\providecommand{\customgenericname}{}
\newcommand{\newcustomtheorem}[2]{%
	\newenvironment{#1}[1]
	{%
		\renewcommand\customgenericname{#2}%
		\renewcommand\theinnercustomgeneric{##1}%
		\innercustomgeneric
	}
	{\endinnercustomgeneric}
}

\theoremstyle{plain}
\newtheorem{lemma}{Lemma}[section]
\newtheorem{proposition}[lemma]{Proposition}
\newtheorem{corollary}[lemma]{Corollary}

\newtheorem{theorem}[lemma]{Theorem}

\theoremstyle{definition}
\newtheorem{definition}[lemma]{Definition}
\newtheorem{remark}[lemma]{Remark}

\newtheorem{example}[lemma]{Example}

\newtheorem*{theorem*}{Theorem}
\newcustomtheorem{customthm}{Theorem}
\newcustomtheorem{customlemma}{Lemma}

\newcommand{\R}{\mathbb{R}}

\newcommand{\N}{\mathbb{N}}

\DeclareSymbolFont{bbold}{U}{bbold}{m}{n}
\DeclareSymbolFontAlphabet{\mathbbold}{bbold}

\renewcommand{\P}{\mathbb{P}}
\newcommand{\E}{\mathbb{E}}

\newcommand{\F}{\mathcal{F}}

\newcommand{\chain}{\mathcal{C}}
\newcommand{\vchain}{\mathbfcal{C}}
\newcommand{\vdist}{\mathbfcal{D}}
\newcommand{\uniform}{\text{U}}
\newcommand{\vuniform}{\textbf{U}}
\newcommand{\ex}[1]{\mathrm{ex}(#1)}
\newcommand{\tail}{\mathcal{T}}

\begin{document}

\title[A Zero-One Law for VMCs]{A Zero-One Law for\\Virtual Markov Chains}
\author{Adam Quinn Jaffe}
\address{Department of Statistics\\
	University  of California\\ 
	367 Evans Hall \#3860\\
	Berkeley, CA 94720-3860 \\
	U.S.A.}
\email{aqjaffe@berkeley.edu}
\thanks{This material is based upon work for which AQJ was supported by the National Science Foundation Graduate Research Fellowship under Grant No. DGE 1752814.}
\date{\today}

\maketitle

\begin{abstract}
A virtual Markov chain (VMC) is a sequence $\{X_N\}_{N=0}^{\infty}$ of Markov chains (MCs) coupled together on the same probability space such that $X_N$ has state space $\{0,1,\ldots, N\}$ and such that removing all instances of $N~+~1$ from the sample path of $X_{N+1}$ results in the sample path of $X_N$ almost surely.
In this paper, we prove an exact characterization of the triviality of the $\sigma$-algebra $\bigcap_{N=0}^{\infty}\sigma(X_N,X_{N+1},\ldots)$.
The main tool for doing this is a decomposition theorem that the $\sigma$-algebra generated by a VMC is equal to the $\sigma$-algebra generated by a certain countably infinite collection of independent constituent MCs.
These constituents are so-called staircase MCs (SMCs), which are defined to be inhomoheneous Markov chains on the non-negative integers which transition only by holding or by jumping to a value equal to the current index.
We also develop some general aspects of the theory of SMCs, including a connection with some classical but very much under-appreciated aspects of convex analysis.
\end{abstract}


\section{Introduction}\label{sec:intro}

The \textit{virtual Markov chains (VMCs)} were introduced in \cite{EvansJaffeVMC} as a projective limit of Markov chains (MCs) on finite state spaces, in the same way that the \textit{virtual permutations} were introduced in \cite{VirtualPermutations} as a projective limit of finite permutations.
Roughly speaking, a VMC is a sequence $\boldsymbol{X} = \{X_N\}_{N\in\N}$ of MCs coupled together on the same probability space such that $X_N$ has state space $\{0,1,\ldots N\}$ and such that removing all instances of $N+1$ from the sample path of $X_{N+1}$ results in the sample path of $X_N$ almost surely.

The indices $N\in\N$ are called \textit{levels}, and the \textit{tail $\sigma$-algebra} with respect to the levels is defined to be
\begin{equation*}
\tail(\boldsymbol{X}) := \bigcap_{N\in\N}\sigma(X_N,X_{N+1},\ldots).
\end{equation*}
This is not the usual $\sigma$-algebra of events depending on arbitrary large times; rather, it is the $\sigma$-algebra of events depending on arbitrary large levels.
The main result of this work (Theorem~\ref{thm:general-ZOL}) is a simple characterization of the triviality of $\tail(\boldsymbol{X})$.

Let us emphasize that this question is not very interesting for certain simple VMCs.
Indeed, observe that a MC $X$ on the state space $\N$ naturally gives rise to a VMC $\boldsymbol{X}=\{X_N\}_{N\in\N}$ by defining the sample path of $X_N$ to be the result of removing all elements of $\{N+1,N+2,\ldots\}$ from the sample path of $X$; a VMC arising in this way is called \textit{classical}.
Then, it is straightforward to show (see Proposition~\ref{prop:CMC-trivial}) that for any classical VMC $\boldsymbol{X}$ the tail $\tail(\boldsymbol{X})$ is trivial if and only if $\boldsymbol{X}$ is itself trivial.
Nonetheless, for most VMCs $\boldsymbol{X}$ are not classical, we demonstrate many examples where $\tail(\boldsymbol{X})$ is trivial but $\boldsymbol{X}$ is non-trivial.

Towards stating the main result, we recall the representation theorem of \cite{EvansJaffeVMC} that the law of any VMC is represented by a pair of a \textit{virtual initial distribution (VID)} $\boldsymbol{\nu}$ and a  \textit{virtual transition matrix (VTM)} $\boldsymbol{K}$.
Roughly speaking, a VID is a sequence $\{\nu_N\}_{N\in\N}$ where each $\nu_N$ is a probability measure on $\{0,1,\ldots N\}$, a VTM is a sequence $\{K_N\}_{N\in\N}$ where each $K_N$ is a transition matrix on the state space $\{0,1,\ldots N\}$.
Moreover, there are appropriate notions of \textit{projectivity} that must be satisfied among the levels of the VID and VTM and of \textit{compatibility} that must be satisfied between the VID and the VTM; we will precisely state these notions later.
We write $\vdist(\boldsymbol{K})$ for the space of VIDs that are compatible with the VTM $\boldsymbol{K}$.
Of course, $\nu_N$ and $K_N$ are nothing more than the initial distribution and transition matrix of $X_N$, respectively.

Let us also recall some aspects of convex analysis discussed in \cite{EvansJaffeVMC}.
It was shown that $\vdist(\boldsymbol{K})$ is a compact convex set, so we can write $\ex{\vdist(\boldsymbol{K})}$ for the set of its extreme points.
It was also shown that for each $a\in\N$ there is a unique element $\delta_a^{\boldsymbol{K}}\in \ex{\vdist(\boldsymbol{K})}$ corresponding to ``starting at state $a$''.
We show in this paper the simple but important fact that for each $a\in\N$ there is also a unique element $\boldsymbol{K}(a,\cdot)\in\vdist(\boldsymbol{K})$ corresponding to the induced distribution ``immediately after'' starting at $a$; it can also be seen as the ``$a$th row'' of $\boldsymbol{K}$ in an appropriate sense.

Now we can (informally) state the zero-one law (which appears as Theorem~\ref{thm:general-ZOL} in the main body).
For the sake of simplicity in this introduction, we focus our attention on the case of a VMC $\boldsymbol{X}$ which is assumed to be \textit{irreducible} in sense that appropriately generalizes the classical sense of irreducibility; we will precisely state this later, but the reader should note that most examples of interest are irreducible.

\begin{customthm}{A}\label{thm:intro-ZOL}
An irreducible VMC $\boldsymbol{X}$ represented by $(\boldsymbol{\nu},\boldsymbol{K})$ has a trivial tail $\sigma$-algebra $\tail(\boldsymbol{X})$ if and only if the conditions
\begin{enumerate}
	\item[(i)] $\boldsymbol{\nu}\in \ex{\vdist(\boldsymbol{K})}$, and
	\item[(ii)] $\boldsymbol{K}(a,\cdot)\in \ex{\vdist(\boldsymbol{K})}$ for all $a\ge 1$
\end{enumerate}
are both satisfied.
\end{customthm}

The key idea of the proof of Theorem~\ref{thm:intro-ZOL} is to recognize the importance of a class non-negative time-inhomogeneous MCs $\boldsymbol{S} = \{S_N\}_{N\in\N}$ with the property that
\begin{equation}\label{eqn:SMC-stay-1}
\frac{\P(S_{N+1} \in \cdot \cap \{0,1,\ldots, N\} \ | \ S_N)}{\P(S_{N+1} \in \{0,1,\ldots, N\} \ | \ S_N)} = \delta_{S_{N}}(\cdot)
\end{equation}
and
\begin{equation}\label{eqn:SMC-jump-1}
\frac{\P(S_{N+1} \in \cdot \setminus \{0,1,\ldots, N\} \ | \ S_N)}{\P(S_{N+1} \notin \{0,1,\ldots, N\} \ | \ S_N)} = \delta_{N+1}(\cdot)
\end{equation}
hold almost surely for all $N\in\N$; in words, $\boldsymbol{S}$ moves only by holding in the current state or by jumping to a value equal to its current index.
The importance of these objects to the study VMCs was, to some degree, known (see \cite[Proposition~3.13]{EvansJaffeVMC}), but in this paper we develop their theory much more comprehensively.
We call these the \textit{staircase MCs (SMCs)} since the graph of any  sample path of such a process is a depiction of an irregular staircase.

By focusing on the role of SMCs, the proof of Theorem~\ref{thm:intro-ZOL} essentially consists of two main ideas:
The first is to notice that any VMC can be naturally decomposed into a infinite collection of independent SMCs, and the second is to notice that $\vdist(\boldsymbol{K})$ is naturally identified with the space of SMCs whose backwards transition probabilities have been fixed.
This second idea is closely related to many well-studied aspects of Martin boundary theory, so the observation brings many powerful tools become to our disposal.

Let us describe the first idea more carefully.
The core of the idea is contained in the following (informal) theorem statement (which appears as Proposition~\ref{prop:decomp-SMC},  Prosposition~\ref{prop:decomp-independence}, and Proposition~\ref{prop:decomp-sigma-alg} in the main body).

\begin{customthm}{B}\label{thm:intro-decomp}
Let $\boldsymbol{X}$ be a VMC.
Let $\boldsymbol{S}^0$ be the sequence of initial positions of $\boldsymbol{X}$ at all levels, and for $a,k\in\N$ let $\boldsymbol{S}^{a,k}$ be the sequence of positions visited immediately after the $k$th visit of $\boldsymbol{X}$ to $a$ at all levels.
Then, $\{\boldsymbol{S}^0\}\cup\{\boldsymbol{S}^{a,k}: a,k\in\N\}$ are independent SMCs, and we have $\sigma(\boldsymbol{X}) = \sigma(\boldsymbol{S}^0)\vee\sigma(\boldsymbol{S}^{a,k}: a,k\in\N)$.
\end{customthm}

The collection $\{\boldsymbol{S}^0\}\cup\{\boldsymbol{S}^{a,k}: a,k\in\N\}$  is called the \textit{staircase decomposition} of $\boldsymbol{X}$.
The rigorous statement of Theorem~\ref{thm:intro-decomp} takes more effort, primarily in terms of notation; we will precisely state it later.
The importance of Theorem~\ref{thm:intro-decomp} is that it guarantees that any probabilistic question about $\boldsymbol{X}$ can be equivalently stated in terms of its staircase decomposition, which has the substantial advantage of its elements being independent.
We also show (Lemma~\ref{lem:SMC-dist}) that if a VMC $\boldsymbol{X}$ is represented by $(\boldsymbol{\nu},\boldsymbol{K})$, then the sequence of marginals of $\boldsymbol{S}^{0}$ is exactly the VID $\boldsymbol{\nu}$, and the sequence of marginals of $\boldsymbol{S}^{a,k}$ is exactly the ``row'' $\boldsymbol{K}(a,\cdot)$ of the VTM $\boldsymbol{K}$.

Now let us return to the main result of Theorem~\ref{thm:intro-ZOL}.
The upshot of this result is that we can deduce the triviality or non-triviality of $\tail(\boldsymbol{X})$ for any VMC $\boldsymbol{X}$ by understanding the extreme point structure of the spaces $\{\vdist(\boldsymbol{K})\}_{\boldsymbol{K}}$ where $\boldsymbol{K}$ ranges over all VTMs.
From the perspective of probability theory, the infinite-dimensional compact convex set $\vdist(\boldsymbol{K})$ is somewhat intimidating: each $\vdist(\boldsymbol{K})$ is the solution set to certain infinite system of balayage inverse problems, generalizing the one-step balayage inverse problems originally studied in \cite{Balayage1,Balayage2}.
Fortunately, from the perspective of convex analysis, this problem has been carefully studied in the elegant but apparently underappreciated paper \cite{Sternfeld} of Sternfeld.
Indeed, Sternfeld shows that the spaces $\{\vdist(\boldsymbol{K})\}_{\boldsymbol{K}}$ are simplices, that essentially any simplex is one of the spaces $\{\vdist(\boldsymbol{K})\}_{\boldsymbol{K}}$, and that many properties of $\vdist(\boldsymbol{K})$ and its extreme point structure can be equivalently cast in terms of easily-checkable properties of $\boldsymbol{K}$.
Thus, combining our zero-one law with Sternfeld's theorems allows one to determine the triviality of non-triviality of the tail $\sigma$-algebra for many VMCs of interest.
Additionally, our work gives a novel probabilistic interpretation to some of Sternfeld's results.

This paper can also be seen as the starting point for understanding the Martin boundary theory of VMCs, which we pursue in forthcoming work.
Indeed, it is easy to show a VMC $\boldsymbol{X} = \{X_N\}_{N\in\N}$ can be viewed as a path-valued MC, where the conditional law of $X_{N+1}$ on $X_{N}$ results from making certain random independent insertions of copies of the character $N+1$ into the sample path of $X_N$.
From this perspective, the zero-one law Theorem~\ref{thm:intro-ZOL} exactly characterizes when a VMC is such that its almost sure limit, in the topology of the Martin compactification, has a trivial law.

The remainder of this paper is structured as follows.
In Section~\ref{sec:background} we review some background material on VMCs and convex analysis that are needed for the rest of the paper.
In Section~\ref{sec:SMC} we develop the general theory of SMCs including the connection with the work of Sternfeld.
In Section~\ref{sec:decomp} we prove the staircase decomposition theorem, and in Section~\ref{sec:tail} we prove the main zero-one law.
Examples are given at the end of each section.

\section{Background}\label{sec:background}

In order to make this paper self-contained, we review in this section some requisite background material.
To begin, set $\N := \{0,1,2,\ldots\}$, and write $\llbracket a,b\rrbracket = \{a,a+1,\ldots, b-1,b\}$ for $a,b\in \N$, which is taken to be empty if $a > b$.
Write $\mathcal{P}(S)$ for the set of Borel probability measures on a Polish space $S$.

\subsection{Virtual Markov Chains}

We begin with the basics of virtual Markov chains (VMCs), and we direct the reader to \cite{EvansJaffeVMC} for proofs of these statements as well as further information on VMCs.
Fix $N\in\N$.
Write
\begin{equation*}
\chain :=\left\{X\in \N^{\N}: \begin{matrix}
\text{ if } X(i) = 0 \text{ for some } i\in\N, \\
\text{then } X(j) = 0 \text{ for } j\ge i
\end{matrix}\right\}
\end{equation*}
and
\begin{equation*}
\chain_N :=\left\{X\in \llbracket 0,N\rrbracket^{\N}: \begin{matrix}
\text{ if } X(i) = 0 \text{ for some } i\in\N, \\
\text{then } X(j) = 0 \text{ for } j\ge 1
\end{matrix}\right\}.
\end{equation*}
For $X\in \chain$ and $A\subseteq \N$, write $I_{X,A}(0) := \inf\{i\in\N: X(i)\in A \}$ and $I_{X,A}(k+1) := \inf\{i\in\N: i> I_{X,A}(k), X(i)\in A \}$ recursively for $k\in\N$.
Using this define the map $P_N:\chain\to\chain_N$ via
\begin{equation*}
(P_N(x))(i) = \begin{cases}
X(I_{X,\llbracket 0,N\rrbracket}(i)), &\text{ if } I_{X,\llbracket 0,N\rrbracket}(i) < \infty, \\
0, &\text{ if } I_{X,\llbracket 0,N\rrbracket}(i) = \infty,
\end{cases}
\end{equation*}
for all $X\in \chain$.
Intuitively speaking, $P_N$ is the map which removes from its input all excursions outside of $\llbracket 0,N\rrbracket$, and it pads the resulting path with 0s in the case that there is a final infinite excursion.
Finally, we define \textit{virtual path space} as the set
\begin{equation*}
\vchain := \left\{\{X_N\}_{N\in\N} \in \prod_{N\in\N}\chain_N: P_{N}(X_{N+1}) = X_N\text{ for all } N\in\N \right\},
\end{equation*}
which is known \cite[Lemma~2.6]{EvansJaffeVMC} to be a compact Polish space.

Now fix some probability space $(\Omega,\F,\P)$.
A $\vchain$-valued random sequence $\boldsymbol{X} =\{X_N\}_{N\in\N}$ is called a \textit{virtual Markov chain (VMC)} if each $X_N$ is a Markov chain (MC) in the state space $\llbracket 0,N\rrbracket$.
We call $X_N$ the \textit{$N$th marginal MC} of $\boldsymbol{X}$, and it follows from this definition that 0 is an absorbing state for each marginal MC.
In words, a VMC is a sequence of MCs on growing state spaces, coupled together so that viewing the $(N+1)$th marginal MC only when it visits states in $\llbracket 0,N\rrbracket$ results in the $N$th marginal MC; in the case that $X_N$ has a last visit to $\llbracket 0,N\rrbracket$, we pad the resulting path with 0s.
It is known that virtual permutations \cite[Example~2.13]{EvansJaffeVMC} and classical MCs in the state space $\N$ with 0 as an absorbing state \cite[Proposition~2.16]{EvansJaffeVMC} both give rise to VMCs in a natural way.
See Figure~\ref{fig:VMC-sample-path} for an illustration of the initial segments of one realization of the sample paths of a VMC.

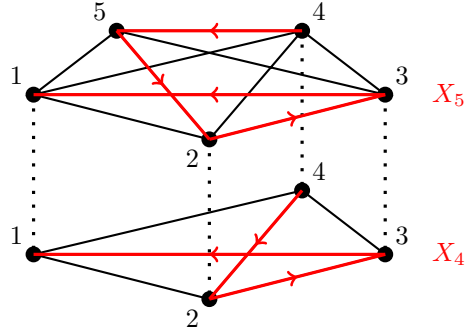
\begin{figure}
	\begin{tikzpicture}[scale=0.85,very thick,decoration={
		markings,
		mark=at position 0.5 with {\arrow{>}}}
	]
	
	\def\XA{0}
	\def\YA{1}
	
	\def\XB{2.75}
	\def\YB{0.3}
	
	\def\XC{5.5}
	\def\YC{1}
	
	\def\XD{4.2}
	\def\YD{2}
	
	\def\XE{1.3}
	\def\YE{2}
	
	\def\offset{2.5}
	
	\coordinate (v1) at (\XA,\YA);
	\coordinate (v2) at (\XB,\YB);
	\coordinate (v3) at (\XC,\YC);
	\coordinate (v4) at (\XD,\YD);
	\coordinate (v5) at (\XE,\YE);
	
	\coordinate (w1) at (\XA,\YA-\offset);
	\coordinate (w2) at (\XB,\YB-\offset);
	\coordinate (w3) at (\XC,\YC-\offset);
	\coordinate (w4) at (\XD,\YD-\offset);
	
	\draw[thick,-] (v1) -- (v2) -- (v3) -- (v4) -- (v5) -- (v1);
	\draw[thick,-] (v2) -- (v4) -- (v1) --  (v3) -- (v5) -- (v2);
	\draw[fill] (v1) circle [radius=0.1] node[above left] {$1$};
	\draw[fill] (v2) circle [radius=0.1] node[below left] {$2$};
	\draw[fill] (v3) circle [radius=0.1] node[above right] {$3$};
	\draw[fill] (v4) circle [radius=0.1] node[above right] {$4$};
	\draw[fill] (v5) circle [radius=0.1] node[above left] {$5$};
	
	\draw[thick,-] (w1) -- (w2) -- (w3) -- (w4) -- (w1);
	\draw[thick,-] (w1) -- (w3);
	\draw[thick,-] (w2) -- (w4);
	\draw[fill] (w1) circle [radius=0.1] node[above left] {$1$};
	\draw[fill] (w2) circle [radius=0.1] node[below left] {$2$};
	\draw[fill] (w3) circle [radius=0.1] node[above right] {$3$};
	\draw[fill] (w4) circle [radius=0.1] node[above right] {$4$};
	
	\draw[loosely dotted,-] (v1) -- (w1);
	\draw[loosely dotted,-] (v2) -- (w2);
	\draw[loosely dotted,-] (v3) -- (w3);
	\draw[loosely dotted,-] (v4) -- (w4);
	
	\draw[postaction={decorate},color=red] (v4)--(v5);
	\draw[postaction={decorate},color=red] (v5)--(v2);
	\draw[postaction={decorate},color=red] (v2)--(v3);
	\draw[postaction={decorate},color=red] (v3)--(v1);
	
	\draw[postaction={decorate},color=red] (w4)--(w2);
	\draw[postaction={decorate},color=red] (w2)--(w3);
	\draw[postaction={decorate},color=red] (w3)--(w1);
	
	\node[red] at (\XC+1,\YC) {$X_5$};
	\node[red] at (\XC+1,\YC-\offset) {$X_4$};
	
	\end{tikzpicture}
	\caption{The initial segments of one realization of the sample paths of a VMC, viewed only at levels 4 and 5.
	The cemetery state of 0 is never visited, so it is not depicted in the state space of each marginal MC.}
	\label{fig:VMC-sample-path}
\end{figure}

Let us elaborate on the way that classical MCs in the state space $\N$ with 0 as a cemetery state give rise to VMCs.
Indeed, if $X$ is such a process, then $\{P_N(X)\}_{N\in\N}$ is a VMC; a VMC arising in this way is called a \textit{classical} VMC, since it is nothing  more than a classical MC ``in disguise''.
It it known (see \cite[Lemma~2.9]{EvansJaffeVMC} and \cite[Proposition~2.16]{EvansJaffeVMC}) that a VMC $\boldsymbol{X} = \{X_N\}_{N\in\N}$ is classical if and only if almost surely $X:=\lim_{N\to\infty}X_N$ exists in $\chain$ and satisfies $\boldsymbol{X}=\{P_N(X)\}_{N\in\N}$.

The study of VMCs can be made concrete by characterizing laws of VMCs through some canonical pieces of data.
To state this, recall that a \textit{virtual initial distribution (VID)} is a sequence $\{\nu_N\}_{N\in\N}$ where each $\nu_N$ is a probability measure on $\llbracket 0,N\rrbracket$, such that the whole collection satisfies $\nu_{N}(a) \ge \nu_{N+1}(a)$ for all $N\in\N$ and $a\in\llbracket 0,N\rrbracket$.
Similarly, a \textit{virtual transition matrix (VTM)} is a sequence $\{K_N\}_{N\in\N}$ where the matrix $K_N\in [0,1]^{(N+1)\times (N+1)}$ is a stochastic matrix satisfying $K_N(0,0) = 1$ and $P_{N}(K_{N+1}) = K_N$ for all $N\in\N$, where we have defined the (non-linear) map
\begin{equation}\label{eqn:K-proj-formula}
P_N\begin{pmatrix}
1 & 0 & 0 \\
w & K & u \\
q & v^{{\text{T}}} & p
\end{pmatrix} := \begin{cases}
\begin{pmatrix}
1 & 0 \\
w + \frac{q}{1-p}u& K + \frac{1}{1-p}uv^{\text{T}} \\
\end{pmatrix}, &\text{ if } p < 1, \\
& \\
\begin{pmatrix}
1 & 0 \\
w + u& K \\
\end{pmatrix}, &\text{ if } p = 1,
\end{cases}
\end{equation}
for $K\in [0,1]^{N\times N}$, $u,v,w\in [0,1]^{N}$, and $p,q\in [0,1]$.
For a VTM $\boldsymbol{K}= \{K_N\}_{N\in\N}$ we also define
\begin{equation}\label{eqn:VTM-balayage}
\pi_{N}^{\boldsymbol{K}}(a) := \begin{cases}
\frac{K_{N+1}(N+1,a)}{1-K_{N+1}(N+1,N+1)}, &\text{ if } K_{N+1}(N+1,N+1) < 1, \\
1, &\text{ else if } a = 0, \\
0, & \text{ else if } a \neq 0.
\end{cases}
\end{equation}
for $a\in \llbracket 0,N\rrbracket$, and it follows that the definition of VTM is equivalent to having  $K_N(a,b) = K_{N+1}(a,b) + K_{N+1}(a,N+1)\pi_{N}^{\boldsymbol{K}}(b)$ for all $a,b\in \llbracket 0,N\rrbracket$.
Furthermore, we say that a VID $\boldsymbol{\nu}= \{\nu_N\}_{N\in\N}$ is \textit{compatible with} $\boldsymbol{K}$, or that the pair $(\boldsymbol{\nu},\boldsymbol{K})$ is \textit{compatible}, if we have $\nu_N(a) = \nu_{N+1}(a) + \nu_{N+1}(N+1)\pi_{N}^{\boldsymbol{K}}(a)$ for all $a\in \llbracket 0,N\rrbracket$.

It is easy to show that if $\boldsymbol{X} = \{X_N\}_{N\in\N}$ is a VMC, then $\boldsymbol{\nu} =\{\nu_N\}_{N\in\N}$ defined via $\nu_N(a) := \P(X_N(0) = a)$ is a VID, that $\boldsymbol{K} = \{K_N\}_{N\in\N}$ defined via $K_N(a,b) := \P(X_N(1) = b \ | \ X_N(0) = a)$ is a VTM, and that the pair $(\boldsymbol{\nu},\boldsymbol{K})$ is compatible.
In this case we say that (the law of) $\boldsymbol{X}$ \textit{is represented by} $(\boldsymbol{\nu},\boldsymbol{K})$.
Also true, but immediately obvious, (see \cite[Theorem~3.22]{EvansJaffeVMC}) is that (the law of) any VMC is represented by some compatible pair $(\boldsymbol{\nu},\boldsymbol{K})$, and that such a representation is unique on the support of $\boldsymbol{X}$.

\subsection{Convex Analysis}

Next we describe some elements of convex analysis, particularly Choquet theory, and we direct the reader to the standard references \cite{Alfsen} and \cite{Phelps} for further information.

To begin, let $E$ be a real Hausdorff locally convex metrizable topological vector space, and let $K$ be a compact convex subset of $E$.
A point $x\in K$ is called \textit{extreme for $K$} if $x',x''\in K$ and $\alpha\in (0,1)$ satisfying $x =  (1-\alpha)x'+\alpha x''$ implies $x=x'=x''$.
We write $\ex{K}$ for the set of points that are extreme for $K$.
We say, for a Borel probability measure $\mu$ on $K$ and a point $x\in K$, that \textit{$\mu$ represents $x$} if we have $\lambda(x) = \int_{K}\lambda\, d\mu$ for all continuous affine functions $\lambda:E\to \R$.

The first fundamental theorem of Choquet theory \cite[Chapter~3]{Phelps} is that each point $x\in K$ is represented by some Borel probability measure $\mu$ on $K$ satisfying $\mu(\ex{K}) =1$, called a \textit{representing probability measure}; since $\ex{K}$ is $G_{\delta}$ hence Borel (see \cite[Proposition~1.3]{Phelps}), it is well-defined to discuss the probability $\mu(\ex{K})$.
In analogy with the finite-dimensional case, one then defines a \textit{simplex} to be a compact convex subset $K$ of $E$ such that each point admits a unique representing probability measure.
The second fundamental theorem of Choquet theory \cite[Chapter~10]{Phelps} is that $K$ is a simplex if and only if a certain lattice-theoretic condition on $A(K)$, the space of continuous affine functions from $K$ to $\R$, is satisfied.
In fact, there are many equivalent formulations of simpliciality, many having strong connections to lattice theory.

In addition to these classical ideas, we will also need some finer results of Choquet theory which appear to be lesser-known to probabilists; more detail on these points and much more can be found in the extremely authoritative monograph \cite{IntRepTheory}.

Our primary interest is in a categorical notion.
That is, we consider the category where the objects are the simplices (in real Hausdorff locally convex metrizable topological vector spaces), and where the morphisms are the continuous affine functions.
By a \textit{projective system} in this category, we mean a partially ordered set $(I,\le)$, a family of simplices $\{K_i\}_{i\in\N}$, and a family of continuous affine maps $\{p_{ij}\}_{i,j\in I, i\le j}$ where $p_{ij}:K_j\to K_i$ is surjective, $p_{ii}$ is the identity on $K_i$, and $p_{ik} = p_{ij}\circ p_{jk}$ holds on $K_k$.
By a \textit{projective limit} of a projective system we mean a simplex $K$ and a family of continuous affine maps $\{q_i\}_{i\in I}$ where $q_i:K\to K_i$ is surjective, and $q_i = p_{ij}\circ q_j$ holds on $K$.
Importantly, it is known (see \cite[Theorem~13]{DaviesVincentSmith}, \cite[Theorem~2]{Jellett}, or \cite[Corollary~12.35]{IntRepTheory}) that projective limits in this category always exist.

Additionally, we will need to understand the extreme point structure of simplices which are projective limits of finite-dimensional simplices; this is exactly the content of the elegant work of Sternfeld \cite{Sternfeld}.
In that paper it is observed that any projective limit of finite-dimensional simplices can be encoded by a lower-triangular matrix called a \textit{representing matrix}, and that many properties of the projective limit can be described in terms of this representing matrix.
We detail these ideas more carefully in Section~\ref{sec:SMC} when they are needed.

\section{Staircase Markov Chains}\label{sec:SMC}

In this section we develop some aspects of a class of inhomogeneous Markov chains that we call \textit{staircase Markov chains (SMCs)}.
Such objects, although without being directly given this name, were introduced and studied in \cite{EvansJaffeVMC} where they were shown to play an important role in the theory of virtual Markov chains (VMCs).
Presently we dedicate a more comprehensive study to SMCs so that their finer probabilistic structure can be used to more fully understand the structure of VMCs.

To begin, define the compact Polish space
\begin{equation*}
\mathbfcal{S} := \left\{\{S_N\}_{N\in\N}\in \prod_{N\in\N}\llbracket 0, N\rrbracket: S_{N+1} \in \{S_N,N+1\} \text{ for all } N\in\N \right\},
\end{equation*}
called the space of \textit{staircases}, which are so named because the graph of any such sequence looks like an irregular staircase.
See Figure~\ref{fig:staircase-process}  for a graph of a sample path of a staircase process.

\begin{figure}
\begin{center}
	\begin{tikzpicture}[scale = 0.5]
	\draw[-,  thick] (0, -1) to (0,6);
	\draw[-,  thick] (-1,0) to (6,0);
	\draw[fill] (0,0) circle [color=red,radius=.2] node[above] {};
	\draw[fill] (1,1) circle [color=red,radius=.2] node[above] {};
	\draw[fill] (2,2) circle [color=red,radius=.2] node[above] {};
	\draw[fill] (3,2) circle [color=red,radius=.2] node[above] {};
	\draw[fill] (4,4) circle [color=red,radius=.2] node[above] {};
	\draw[fill] (5,4) circle [color=red,radius=.2] node[above] {};
	\end{tikzpicture} 
\end{center}
\caption{The graph of a sample path of a staircase process.}
\label{fig:staircase-process}
\end{figure}
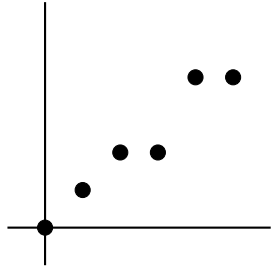

A probability measure $\mu\in \mathcal{P}(\mathbfcal{S})$ is called a \textit{staircase process} and a staircase process which is an inhomogeneous Markov chain (MC) is called a \textit{staircase Markov chain (SMC)}.
Equivalently, a SMC is the law $\mu$ of an $\N$-valued inhomogeneous MC $\{S_N\}_{N\in\N}$ such that
\begin{equation}\label{eqn:SMC-stay}
\frac{\mu(S_{N+1} \in \cdot \cap \llbracket 0, N\rrbracket \ | \ S_N)}{\mu(S_{N+1} \in \llbracket 0, N\rrbracket \ | \ S_N)} = \delta_{S_{N}}(\cdot)
\end{equation}
and
\begin{equation}\label{eqn:SMC-jump}
\frac{\mu(S_{N+1} \in \cdot \setminus \llbracket 0, N\rrbracket \ | \ S_N)}{\mu(S_{N+1} \notin \llbracket 0, N\rrbracket \ | \ S_N)} = \delta_{N+1}(\cdot)
\end{equation}
hold $\mu$-almost surely for all $N\in\N$.
We write $\mathcal{P}_{\mathrm{MC}}(\mathbfcal{S})$ for the space of all SMCs, which is compact.
By a slight abuse of notation, an $\mathbfcal{S}$-valued random variable whose law is an SMC is also called an SMC.

While the space $\mathcal{P}_{\mathrm{MC}}(\mathbfcal{S})$ appears rather complicated, it was observed in \cite{EvansJaffeVMC} that it can be reparameterized in a more convenient way.
To do this, for $\mu_N\in \mathcal{P}(\llbracket 0,N\rrbracket)$ and $\mu_{N+1}\in \mathcal{P}(\llbracket 0,N+1\rrbracket)$, write $\mu_{N+1}\le \mu_N$ to mean that we have $\mu_{N+1}(a)\le \mu_N(a)$ for all $a\in \llbracket 0,N\rrbracket$.
Then define
\begin{equation*}
\vdist := \left\{\{\nu_N\}_{N\in\N}\in \prod_{N\in\N}\mathcal{P}(\llbracket 0,N\rrbracket): \nu_{N+1}\le \nu_{N} \text{ for all } N\in\N \right\},
\end{equation*}
which is clearly compact and convex.
Now we recall the following result, which is an  equivalent formulation of \cite[Proposition~3.13]{EvansJaffeVMC}:

\begin{proposition}\label{prop:SMC-to-marginals}
	The map $\Phi:\mathcal{P}_{\mathrm{MC}}(\mathbfcal{S})\to \vdist$ sending each SMC to its sequence of marginal distributions is a homeomorphism.
	More specifically, for any $\{\nu_N\}_{N\in\N}\in\vdist$, if we define $\mu\in\mathcal{P}_{\mathrm{MC}}(\mathbfcal{S})$ via $\mu\circ S_0^{-1} = \nu_0$ and
	\begin{equation}\label{eqn:marginal-to-SMC}
	\mu(S_{N+1} \in \cdot \ | \ S_N) =\frac{\nu_{N+1}(S_N)}{\nu_{N}(S_N)}\delta_{S_N}(\cdot) + \left(1-\frac{\nu_{N+1}(S_N)}{\nu_{N}(S_N)}\right)\delta_{N+1}(\cdot)
	\end{equation}
	for all $N\in\N$, then $\Phi(\mu) = \{\nu_N\}_{N\in\N}$.
\end{proposition}

Since $\vdist$ is compact and convex, it is natural to look for an integral representation theorem expressing each point via a representing measure supported on the extreme points.
In fact, the correspondence $\Phi$ provides a somewhat canonical way to do this.
Note for all $\boldsymbol{S}\in \mathbfcal{S}$ that $\Phi(\delta_{\boldsymbol{S}})$ lies in $\vdist$ and is an extreme point of $\vdist$.
Thus, defining the map $\phi:\mathbfcal{S}\to\vdist$ via $\phi(\boldsymbol{S}) := \Phi(\delta_{\boldsymbol{S}})$ for $\boldsymbol{S}\in\mathbfcal{S}$, we have the following:

\begin{proposition}\label{prop:SMC-int-rep}
	Each $\boldsymbol{\nu}\in \vdist$ is represented by $(\Phi^{-1}(\boldsymbol{\nu}))\circ \phi^{-1} \in \mathcal{P}(\ex{\vdist})$.
\end{proposition}

\begin{proof}
	By Proposition~\ref{prop:SMC-to-marginals} it is enough to show for all $\mu\in \mathcal{P}_{\mathrm{MC}}(\mathbfcal{S})$ and all continuous affine functions $\lambda:\vdist\to\R$ that we have
	\begin{equation}\label{eqn:SMC-int-rep}
	\lambda(\Phi(\mu)) = \int_{\mathbfcal{S}}\lambda(\phi(\boldsymbol{S}))\, d\mu(\boldsymbol{S}).
	\end{equation}
	To do this, we define the maps $T_M:\vdist\to \vdist$ for $M\in\N$ by setting $T_M(\{\nu_N\}_{N\in\N}) = \{\nu_{\min\{N,M\}}\}_{N\in\N}$.
	In particular, $\{T_M\}_{M\in\N}$ are continuous and affine and satisfy $T_M(\boldsymbol{\nu})\to \boldsymbol{\nu}$ as $M\to\infty$.
	Now let $\lambda:\vdist\to\R$ be any continuous affine function, and set $B := \sup_{\boldsymbol{\nu}\in \vdist}|\lambda(\boldsymbol{\nu})| < \infty$, as well as $\lambda_M:\vdist\to\R$  for $M\in\N$ via $\lambda_M:=\lambda \circ T_M$.
	Then, observe that $\{\lambda_M\}_{M\in\N}$ are continuous affine maps that are uniformly bounded by $B$ and that satisfy $\lambda_M(\boldsymbol{\nu})\to \lambda(\boldsymbol{\nu})$ as $M\to\infty$.
	Therefore, by dominated convergence, it is enough to prove \eqref{eqn:SMC-int-rep} for continuous affine functions that depend only on finitely many coordinates of $\vdist$.
	
	Indeed, for $M\in\N$ write $F_M$ for the space of continuous affine functions $\lambda:\vdist\to\R$ that depend on $\{\nu_N\}_{N\in\N}$ only through $\{\nu_N\}_{N\le M}$.
	Then define the the point-evaluations $e_{a,N}:\vdist\to\R$ via $e_{a,N}(\boldsymbol{\nu}) := \nu_N(a)$, and note that $F_M$ is spanned by $\{e_{a,N}: N\le M, a\in \llbracket 0,N\rrbracket \}$.
	Since we clearly have
	\begin{equation*}
	e_{a,N}(\Phi(\mu)) = \int_{\mathbfcal{S}}e_{a,N}(\phi(s))\, d\mu(s)
	\end{equation*}
	for all $N\in\N$ and $a\in\llbracket 0,N\rrbracket$, the result follows.
\end{proof} 

Next we identify certain sub-objects of $\mathcal{P}_{\mathrm{MC}}(\mathbfcal{S})$ whose structure is of central importance.
To do this, fix an element $\boldsymbol{\pi}=\{\pi_N\}_{N\in\N}\in \prod_{N\in\N}\mathcal{P}(\llbracket 0,N\rrbracket)$, which we call a \textit{balayage}.
For any balayage $\boldsymbol{\pi}$, we define the space $\mathcal{P}_{\mathrm{MC}}^{\boldsymbol{\pi}}(\mathbfcal{S})$ to consist of all $\mu\in \mathcal{P}_{\mathrm{MC}}(\mathbfcal{S})$ such that we have $\mu(S_N\in \cdot \ | \ S_{N+1} = N+1) = \pi_N(\cdot)$ for all $N\in\N$ holding $\mu$-almost surely.
In words, $\mathcal{P}_{\mathrm{MC}}^{\boldsymbol{\pi}}(\mathbfcal{S})$ is the space of all laws of SMCs whose backwards transition probabilities are given by $\boldsymbol{\pi}$.
Importantly, we have the following classical result (see the monograph \cite{Preston}, and, in particular Theorem~2.1):

\begin{theorem}\label{thm:SMC-tail-ext}
For any balayage $\boldsymbol{\pi}$, the space $\mathcal{P}_{\mathrm{MC}}^{\boldsymbol{\pi}}(\mathbfcal{S})$ is a simplex.
Moreover, a SMC $\mu\in\mathcal{P}_{\mathrm{MC}}^{\boldsymbol{\pi}}(\mathbfcal{S})$ satisfies $\mu\in \ex{\mathcal{P}_{\mathrm{MC}}^{\boldsymbol{\pi}}(\mathbfcal{S})}$ if and only if its tail $\sigma$-algebra defined via $\bigcap_{N\in\N}\sigma(S_N,S_{N+1},\ldots)$ is trivial.
\end{theorem}

In the abstract setting, it can be hard to reason directly about whether we have $\mu\in \ex{\mathcal{P}_{\mathrm{MC}}^{\boldsymbol{\pi}}(\mathbfcal{S})}$, but we now show how the correspondence with $\vdist$ can simplify this task greatly.

Crucially, we note that, although $\mathcal{P}_{\mathrm{MC}}(\mathbfcal{S})$ is not itself convex, the restriction $\Phi:C\to\vdist$ is affine for any convex subset $C\subseteq\mathcal{P}_{\mathrm{MC}}(\mathbfcal{S})$.
In particular, Proposition~\ref{prop:SMC-to-marginals} implies that $\Phi:\mathcal{P}^{\boldsymbol{\pi}}_{\mathrm{MC}}(\mathbfcal{S})\to\vdist$ is continuous and affine for any balayage $\boldsymbol{\pi}$.
In fact, we can say more.
Define
\begin{equation*}
\vdist(\boldsymbol{\pi}) := \{\{\nu_N\}_{N\in\N}\in \vdist: \nu_{N} = \nu_{N+1} + \nu_{N+1}(N+1)\pi_N \text{ for all } N\in\N \},
\end{equation*}
and note the following.

\begin{lemma}\label{lem:SMC-balayage-homeo}
For any balayage $\boldsymbol{\pi}$, the map $\Phi:\mathcal{P}^{\boldsymbol{\pi}}_{\mathrm{MC}}(\mathbfcal{S})\to\vdist(\boldsymbol{\pi})$ is an affine homeomorphism.
\end{lemma}

\begin{proof}
First let us show that $\mu\in\mathcal{P}^{\boldsymbol{\pi}}_{\mathrm{MC}}(\mathbfcal{S})$ implies $\Phi(\mu)\in \vdist(\boldsymbol{\pi})$.
Indeed, for $N\in\N$ and $a\in \llbracket 0,N\rrbracket$, we can compute:
\begin{equation*}
\begin{split}
\nu_{N+1}&(N+1)\pi_N(a) + \nu_{N+1}(a) \\
&= \mu(S_{N+1} = N+1)\mu(S_N = a \ | \ S_{N+1} = N+1) + \mu(S_{N+1} = a) \\
&= \mu(S_N =  a,S_{N+1} = N+1) + \mu(S_{N+1} = a) \\
&= \mu(S_N =  a,S_{N+1} = N+1) + \mu(S_N = a,S_{N+1} = a) \\
&= \mu(S_N =  a) \\
&= \nu_N(a),
\end{split}
\end{equation*}
as needed.
Next note by Proposition~\ref{prop:SMC-to-marginals} that $\Phi:\mathcal{P}^{\boldsymbol{\pi}}_{\mathrm{MC}}(\mathbfcal{S})\to\vdist(\boldsymbol{\pi})$ is a continuous injection.
Thus, it suffices to show that $\boldsymbol{\nu} =\{\nu_N\}_{N\in\N}\in\vdist(\boldsymbol{\pi})$ implies $\Phi^{-1}(\boldsymbol{\nu})\in\mathcal{P}^{\boldsymbol{\pi}}_{\mathrm{MC}}(\mathbfcal{S})$.
To do this, take $N\in\N$ and $a\in \llbracket 0,N\rrbracket$, and compute:
\begin{equation*}
\begin{split}
\mu(S_N &= a,S_{N+1} = N+1) \\
&= \mu(S_{N+1} =  N+1 \ | \ S_N =  a)\mu(S_N =  a) \\
&= \left(1-\frac{\nu_{N+1}(a)}{\nu_N(a)}\right)\nu_N(a) \\
&= \nu_N(a)-\nu_{N+1}(a) \\
&= \nu_{N+1}(N+1)\pi_N(a) \\
&= \mu(S_{N+1} = N+1)\pi_N(a).
\end{split}
\end{equation*}
Thus, dividing both sides by $\mu(S_{N+1} = N+1)$
yields
\begin{equation*}
\mu(S_N = a\ | \ S_{N+1} = N+1) = \pi_N(a)
\end{equation*}
as claimed.
\end{proof}

The following simple but important result establishes that a SMC with prescribed backwards transition probabilities can be recovered even when a finite number of its marginal distributions are forgotten.

\begin{lemma}\label{lem:SMC-marginals-residual}
	Take any balayage $\boldsymbol{\pi} = \{\pi_N\}_{N\in\N}$, and any $a\in\N$ and $\{\nu_N\}_{N\ge a}\in\prod_{N\ge a}\mathcal{P}(\llbracket 0,N\rrbracket)$.
	If we have
	\begin{equation*}
	\nu_N(a') = \nu_{N+1}(a') + \nu_{N+1}(N+1)\pi_N
	\end{equation*}
	for all $N\ge a$ and all $a'\in \llbracket 0,N\rrbracket$, then there is a unique $\{\tilde{\nu}_N\}_{N\in\N}\in \vdist(\boldsymbol{\pi})$ with $\tilde{\nu}_N = \nu_N$ for all $N\ge a$.
\end{lemma}

\begin{proof}
If $\{\tilde{\nu}_N\}_{N\in\N}\in \vdist(\boldsymbol{\pi})$ has $\tilde{\nu}_a =  \nu_a$, then, by backwards induction on $N<a$, it follows that we have $\nu_N(a') = \nu_{N+1}(a') + \nu_{N+1}(N+1)\pi_{N}^{\boldsymbol{K}}(a')$ for all $a'\in \llbracket 0,N\rrbracket$.
Since this is in fact an element of  $\vdist(\boldsymbol{\pi})$, the result follows.
\end{proof}

It is straightforward to see that the element $\boldsymbol{0} := \{\delta_0\}_{N\in\N}$ has $\boldsymbol{0}\in\vdist(\boldsymbol{\pi})$ for all balayages $\boldsymbol{\pi}$.
Moreover, if balayages $\boldsymbol{\pi} = \{\pi_N\}_{N\in\N}$ and $\boldsymbol{\pi}' = \{\pi_N'\}_{N\in\N}$ are such that $\pi_N \neq \pi_N'$ holds for all $N\ge 1$, then it follows that  we have $\vdist(\boldsymbol{\pi})\cap \vdist(\boldsymbol{\pi}') =  \{\boldsymbol{0}\}$.
We believe it would be interesting, although tangential to this work, to undertake a more careful study of the overlap between the regions $\{\vdist(\boldsymbol{\pi})\}_{\boldsymbol{\pi}}$ where  $\boldsymbol{\pi}$ ranges over all balayages.
Also see Figure~\ref{fig:SMC-marginals} for an illustration summarizing the geometric aspects of this correspondence between SMCs and their  marginal distributions.

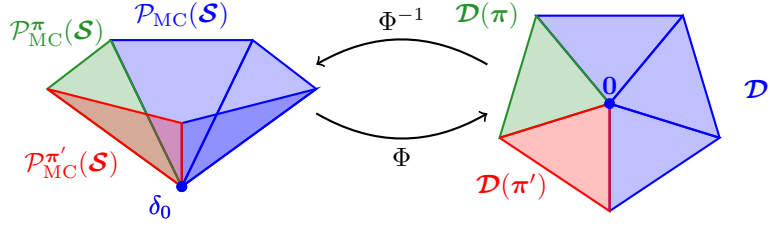
\begin{figure}
\begin{tikzpicture}[scale=0.65]
\definecolor{Green}{HTML}{27912b}

\coordinate (b) at (2.75,-1);
\coordinate (v1) at (0,1);
\coordinate (v2) at (2.75,0.3);
\coordinate (v3) at (5.5,1);
\coordinate (v4) at (4.2,2);
\coordinate (v5) at (1.3,2);

\draw[fill,fill opacity=0.25,color=blue,thick,-] (v2) -- (v3) -- (b)  -- (v2);
\draw[fill,fill opacity=0.25,color=blue,thick,-] (v3) -- (v4) -- (b)  -- (v3);
\draw[fill,fill opacity=0.25,color=blue,thick,-] (v4) -- (v5) -- (b)  -- (v4);
\draw[fill,fill opacity=0.25,color=Green,thick,-] (v5) -- (v1) -- (b)  -- (v5);
\draw[fill,fill opacity=0.25,color=red,thick,-] (v1) -- (v2) -- (b)  -- (v1);

\draw[fill,color=blue] (b) circle [radius=0.1] node[below left] {$\delta_{0}$};

\node[red] at (0.5,-0.5) {$\mathcal{P}_{\mathrm{MC}}^{\boldsymbol{\pi}'}(\mathbfcal{S})$};
\node[Green] at (0.2,2.15) {$\mathcal{P}_{\mathrm{MC}}^{\boldsymbol{\pi}}(\mathbfcal{S})$};
\node[blue] at (2.75,2.5) {$\mathcal{P}_{\mathrm{MC}}(\mathbfcal{S})$};

\coordinate (c) at (11.5,0.7);
\coordinate (w1) at (9.25,0);
\coordinate (w2) at (11.5,-1.5);
\coordinate (w3) at (13.75,0);
\coordinate (w4) at (13,2.5);
\coordinate (w5) at (10,2.5);

\draw[fill,fill opacity=0.25,color=blue,thick,-] (w2) -- (w3) -- (c)  -- (w2);
\draw[fill,fill opacity=0.25,color=blue,thick,-] (w3) -- (w4) -- (c)  -- (w3);
\draw[fill,fill opacity=0.25,color=blue,thick,-] (w4) -- (w5) -- (c)  -- (w4);
\draw[fill,fill opacity=0.25,color=Green,thick,-] (w5) -- (w1) -- (c)  -- (w5);
\draw[fill,fill opacity=0.25,color=red,thick,-] (w1) -- (w2) -- (c)  -- (w1);

\draw[fill,color=blue] (b) circle [radius=0.1] node[below left] {$\delta_{0}$};

\node[red] at (9.5,-1) {$\vdist(\boldsymbol{\pi}')$};
\node[Green] at (9,2.5) {$\vdist(\boldsymbol{\pi})$};
\node[blue] at (14.5,1) {$\vdist$};

\draw[fill,color=blue] (c) circle [radius=0.1] node[above] {$\boldsymbol{0}$};

\path[thick, ->] (5.5,0.5) edge[bend right] node [below] {$\Phi$} (9,0.5);
\path[thick, ->] (9,1.5) edge[bend right] node [above] {$\Phi^{-1}$} (5.5,1.5);
\end{tikzpicture}
\caption{A cartoon depicting the correspondence between SMCs and their sequences of marginal distributions.
Here, $\mathcal{P}_{\mathrm{MC}}(\mathbfcal{S})$ is the space of all SMCs (compact and non-convex), $\vdist$ is the space of all sequences of marginal distributions of SMCs (compact and convex), and $\Phi$ is a homeomorphism.
We have coverings (not bona fide partitions) of $\mathcal{P}_{\mathrm{MC}}(\mathbfcal{S})$ and $\vdist$ by collections of simplices $\{\mathcal{P}^{\boldsymbol{\pi}}_{\mathrm{MC}}(\mathbfcal{S})\}_{\boldsymbol{\pi}}$ and $\{\vdist(\boldsymbol{\pi})\}_{\boldsymbol{\pi}}$, respectively, where $\boldsymbol{\pi}$ ranges over all balayages.
Finally, we have that $\mathcal{P}^{\boldsymbol{\pi}}_{\mathrm{MC}}(\mathbfcal{S})$ and $\vdist(\boldsymbol{\pi})$ are affinely homeomorphic for all $\boldsymbol{\pi}$, where the affine homeomorphism is just the restriction of $\Phi$ to $\mathcal{P}^{\boldsymbol{\pi}}_{\mathrm{MC}}(\mathbfcal{S})$.}
\label{fig:SMC-marginals}
\end{figure}

As a consequence of Theorem~\ref{thm:SMC-tail-ext} and Lemma~\ref{lem:SMC-balayage-homeo}, we now have the following result characterizing the triviality of the tail $\sigma$-algebra of a SMC.
This is the first step on the way to characterizing the triviality of the tail of a VMC.

\begin{corollary}\label{cor:SMC-tail-ext}
A SMC $\mu\in\mathcal{P}^{\boldsymbol{\pi}}_{\mathrm{MC}}(\mathbfcal{S})$ has trivial tail if and only if $\Phi(\mu)\in \ex{\vdist(\boldsymbol{\pi})}$.
\end{corollary}

Presumably, the value of Corollary~\ref{cor:SMC-tail-ext} over Theorem~\ref{thm:SMC-tail-ext} is that it is somehow easier to check for extremality in $\vdist(\boldsymbol{\pi})$ than in $\mathcal{P}_{\mathrm{MC}}^{\boldsymbol{\pi}}(\mathbfcal{S})$.
The remainder of this section is spent verifying this, and, more precisely, showing convex analysis provides a robust set of tools for determining extremality in $\vdist(\boldsymbol{\pi})$.

To begin, note that the simpliciality of $\mathcal{P}^{\boldsymbol{\pi}}_{\mathrm{MC}}(\mathbfcal{S})$ and Lemma~\ref{lem:SMC-balayage-homeo} together imply that $\vdist(\boldsymbol{\pi})$ is a simplex, for any balayage $\boldsymbol{\pi}$.
However, it is instructive to also consider an alternative way to prove this:
For each $N\in\N$, define the map $p_N:\mathcal{P}(\llbracket0,N+1\rrbracket)\to \mathcal{P}(\llbracket0,N\rrbracket)$ via $p_N(\mu_{N+1}) :=  \mu_{N+1} + \mu_{N+1}(N+1)\pi_N$ for $\mu_{N+1}\in\mathcal{P}(\llbracket0,N+1\rrbracket)$.
It is clear that $p_N$ is an affine continuous surjection.
Moreover, the collection $(\{\mathcal{P}(\llbracket0,N\rrbracket)\}_{N\in\N},\{p_N\})_{N\in\N}$ is a projective system in the category whose objects are simplices and whose morphisms are continuous affine maps, and its projective limit is exactly $\vdist(\boldsymbol{\pi})$.
Since it is known that projective limits in this category are themselves simplices (see \cite[Theorem~13]{DaviesVincentSmith}, \cite[Theorem~2]{Jellett}, or \cite[Corollary~12.35]{IntRepTheory}), it follows that $\vdist(\boldsymbol{\pi})$ is a simplex.

The reason for this framing is that the extreme point structure of such projective limits of simplices is well-understood.
Indeed, let us situate our work in the context of the elegant, yet apparently under-appreciated, work of Sternfeld \cite{Sternfeld}:
For each $N\in\N$, the maps $\{e_{a,N}: a\in \llbracket 0,N\rrbracket\}$ defined via $e_{a,N}(\boldsymbol{\nu}) =  \nu_N(a)$ for $\boldsymbol{\nu} = \{\nu_N\}_{N\in\N}\in\vdist(\boldsymbol{\pi})$ form a \textit{peaked partition of unity} in $\vdist(\boldsymbol{\pi})$, and they also satisfy
\begin{equation*}
e_{a,N} = e_{a,N+1} + \pi_N(a)e_{N+1,N+1}
\end{equation*}
for all $a\in \llbracket 0,N\rrbracket$.
As such, the balayge probabilities $\{\pi_N\}_{N\in\N}$ are a \textit{representing matrix} for the simplex $\vdist(\boldsymbol{\pi})$.
From this representing matrix we can characterize many properties of $\vdist(\boldsymbol{\pi})$ itself.
For example, it follows from \cite[Proposition~3.1]{Sternfeld} that for each $a\in\N$ there exists a unique element $\delta_a^{\boldsymbol{\pi}}\in\ex{\vdist(\boldsymbol{\pi})}$ such that $e_{a,N}(\delta_a^{\boldsymbol{\pi}}) = 1$ holds for all $N\ge a$.
Moreover, we have \cite[Proposition~3.4]{Sternfeld}:

\begin{proposition}
For any balayage $\boldsymbol{\pi}$, the sequence $\{\delta_a^{\boldsymbol{\pi}}\}_{a\in\N}$ is dense in $\ex{\vdist(\boldsymbol{\pi})}$. 
\end{proposition}

\begin{remark}\label{rem:states-compactification}
In light of the preceding result, it seems reasonable to think of the set $\ex{\vdist(\boldsymbol{\pi})}$, which need not  be compact, as a sort of ``compactification'' of $\N$.
Indeed, the proof of Theorem~\ref{thm:SMC-tail-ext} relies on the construction of a common Martin compactification for all of the MCs in the class $\mathcal{P}_{\mathrm{MC}}^{\boldsymbol{\pi}}(\mathbfcal{S})$, so it follows that $\ex{\mathcal{P}_{\mathrm{MC}}^{\boldsymbol{\pi}}(\mathbfcal{S})}$ is essentially the uncountable union of all of these compact spaces.
By Lemma~\ref{lem:SMC-balayage-homeo}, this idea transfers to $\ex{\vdist(\boldsymbol{\pi})}$.
Thus, one should think of $\ex{\vdist(\boldsymbol{\pi})}\setminus \{\delta_a^{\boldsymbol{\pi}}\}_{a\in\N}$ as states ``at infinity''.
\end{remark}

If a balayage $\boldsymbol{\pi}$ is such that $\ex{\vdist(\boldsymbol{\pi})}$ is compact, then a SMC $\mu\in \mathcal{P}_{\mathrm{MC}}^{\boldsymbol{\pi}}(\mathbfcal{S})$ has trivial tail $\sigma$-algebra if and only if there is a sequence $\{a_N\}_{N\in\N}$ in $\N$ with $\delta_{a_N}^{\boldsymbol{\pi}} \to \Phi(\mu)$.
Thus, in this case, the task of finding all extreme points of $\vdist(\boldsymbol{\pi})$ can be reduced to the task of computing $\{\delta_a^{\boldsymbol{\pi}}\}_{a\in\N}$ and all subsequential limits thereof.
While this is quite difficult to do abstractly, it is often quite easy for most concrete examples of interest.
We will see many examples of this later.

Thus, it remains to try to understand which balayages $\boldsymbol{\pi} =\{\pi_N\}_{N\in\N}$ are such that $\ex{\vdist(\boldsymbol{\pi})}$ is compact.
Fortunately, this question is also answered by Sternfeld.
To state the answer, define $\pi_{a,N}\in \mathcal{P}(\llbracket 0,N\rrbracket)$ for $a,N\in\N$ via the recursion
\begin{equation*}
\pi_{a,N} = \begin{cases}
\delta_{a}, &\text{ if } a\le N, \\
\sum_{b=0}^{a-1}\pi_{a-1}(b)\pi_{a',N}, &\text{ if } a> N,
\end{cases}
\end{equation*}
and note that we have $\pi_{N+1,N} = \pi_N$ for all $N\in\N$.
More generally, for all $a,N\in\N$, the probability measure $\pi_{a,N}$ represents the law of the first index in $\llbracket 0,N\rrbracket$ which is hit by a particle starting at $a$ which moves backwards according to the balayage probabilities $\{\pi_N\}_{N\in\N}$.
We refer to $\{\pi_{a,N}\}_{a,N\in\N}$ as the \textit{extended balayage probabilities} of $\boldsymbol{\pi}$.
Finally, we have $\pi_{a,N}(b) =  e_{b,N}(\delta_a^{\boldsymbol{\pi}})$ for all $a,N\in\N$ and $a'\in \llbracket 0,N\rrbracket$ by \cite[Proposition~3.2]{Sternfeld}.
Now we can use the following result which appears as \cite[Theorem~4.2]{Sternfeld}:

\begin{theorem}\label{thm:Sternfeld-ext-cpt}
	The set $\ex{\vdist(\boldsymbol{\pi})}$ is compact if and only if for each $M\in\N$ and $c\in \llbracket  0,M\rrbracket$ we have
	\begin{equation*}
	\sum_{b = 0}^{N} (\pi_{b,M}(c))^2\pi_{a,N}(b) \to (\pi_{a,M}(c))^2
	\end{equation*}
	uniformly in $a\in\N$ as $N\to\infty$.
\end{theorem}

Observe that the two sides above are equal if $a\in\N$ has $a\le N$, hence we only need to consider the limit uniformly over $a\in\N$ with $a\ge N$ as $N\to\infty$.
It is also instructive to also give two more results of Sternfeld.
These appear, respectively, as \cite[Theorem~4.1 and Theorem~4.3]{Sternfeld}.

\begin{theorem}\label{thm:Sternfeld-ext-balayage}
	If $\boldsymbol{\pi} = \{\delta_{a_N}\}_{N\in\N}$ for some sequence $\{a_N\}_{N\in\N}\in  \prod_{N\in\N}\llbracket 0,N\rrbracket$, then $\ex{\vdist(\boldsymbol{\pi})}$ is compact and totally disconnected.
\end{theorem}

\begin{theorem}\label{thm:Sternfeld-ext-K0}
	Suppose the set $\ex{\vdist(\boldsymbol{\pi})}$ is compact.
	Then, we have $\#(\ex{\vdist(\boldsymbol{\pi})}\setminus \{\delta_a^{\boldsymbol{\pi}}\}_{a\in\N}) = 1$ if and only if $\{\pi_{a,N}(N)\}_{a\in\N}$ converges for each $N\in\N$.
\end{theorem}

By combining the probabilistic aspects of the first part of this section with the convexity theory aspects of the second part of this section, we can analyze some important examples.
In particular, we compute $\mathcal{P}_{\mathrm{MC}}^{\boldsymbol{\pi}}(\mathbfcal{S})$ for various balayages $\boldsymbol{\pi}$ and are able to determine which elements thereof have trivial tail.
Many of these examples were previously treated in \cite[Subsection~4.3]{EvansJaffeVMC}; the previous a hands-on approach was only able to lead to partial results, but we now use Sternfeld's theorems to complete the analyses.

For the sake of clarity, we choose to depict elements of $\vdist$ visually as infinite lower triangular matrices.
Additionally, we make a few other simplifications.
First, for $a\in\N$, when rows $0,\ldots a-1$ of $\delta_a^{\boldsymbol{\pi}}$ have already been depicted, we will omit the ellipses and it will be understood that all remaining rows are equal to $\delta_a$.
Second, we will use matrices with square brackets to denote that the 0th row and column have been omitted; this is convenient since the entry in the 0th column can always be deduced from the other entries, and since it is often equal to 0 itself.
For example, for any balayage $\boldsymbol{\pi}$, the element $\delta_0^{\boldsymbol{\pi}} = \boldsymbol{0}\in\vdist(\boldsymbol{\pi})$  can be visualized as follows:
\begin{equation*}
\boldsymbol{0} = \delta_0^{\boldsymbol{\pi}} = \begin{pmatrix}
1 & & & & & \\
1 & 0 & & & & \\
1 & 0 & 0 & & & \\
1 & 0 & 0 & 0 & & \\
1 & 0 & 0 & 0 & 0 & \\
\vdots & \vdots & \vdots & \vdots & \vdots & \ddots \\
\end{pmatrix} = \begin{bmatrix}
0 & & & & \\
0 & 0 & & & \\
0 & 0 & 0 & & \\
0 & 0 & 0 & 0 & \\
\vdots & \vdots & \vdots & \vdots & \ddots \\
\end{bmatrix}
\end{equation*}
Now we move on to the examples of interest.

\begin{example}\label{ex:down-balayage}
Consider the balayage $\boldsymbol{\pi} = \{\delta_{N}\}_{N\in\N}$.
By Theorem~\ref{thm:Sternfeld-ext-balayage}, the space $\ex{\vdist(\boldsymbol{\pi})}$ is compact and totally disconnected.
In fact, we directly compute $\delta_a^{\boldsymbol{\pi}} = \{\delta_{\min\{N,a\}}\}_{N\in\N}$ for $a\in\N$.
Visually, we can depict $\delta_1^{\boldsymbol{\pi}},\delta_2^{\boldsymbol{\pi}},\delta_3^{\boldsymbol{\pi}}$, and $\delta_4^{\boldsymbol{\pi}}$ as
\begin{equation*}
\begin{bmatrix}
1 & & & \\
1 & 0 & & \\
1 & 0 & 0 & \\
1 & 0 & 0 & 0 \\
\end{bmatrix}, \quad \begin{bmatrix}
1 & & & \\
0 & 1 & & \\
0 & 1 & 0 & \\
0 & 1 & 0 & 0 \\
\end{bmatrix}, \quad \begin{bmatrix}
1 & & & \\
0 & 1 & & \\
0 & 0 & 1 & \\
0 & 0 & 1 & 0 \\
\end{bmatrix}, \text{ and }\begin{bmatrix}
1 & & & \\
0 & 1 & & \\
0 & 0 & 1 & \\
0 & 0 & 0 & 1 \\
\end{bmatrix}.
\end{equation*}
From here we see that $\{\delta_a^{\boldsymbol{\pi}}\}_{a\in\N}$ converges as $a\to\infty$ to $\{\delta_N\}_{N\in\N}$, or, visually
\begin{equation*}
\lim_{a\to\infty}\delta_a^{\boldsymbol{\pi}} = \begin{bmatrix}
1 & & & & \\
0 & 1 & & & \\
0 & 0 & 1 & & \\
0 & 0 & 0 & 1 & \\
\vdots & \vdots & \vdots & \vdots & \ddots \\
\end{bmatrix}.
\end{equation*}
In particular, we have shown
\begin{equation*}
\ex{\vdist(\boldsymbol{\pi})} = \{\delta_{a}^{\boldsymbol{\pi}}\}_{N\in\N}: a\in\N \}\cup\left\{\lim_{a\to\infty}\delta_a^{\boldsymbol{\pi}}\right\},
\end{equation*}
which is the same conclusion completed by the more hands-on analysis of \cite[Example~4.13]{EvansJaffeVMC}.
\end{example}

\begin{example}\label{ex:two-down-balayage}
Consider the balayage $\boldsymbol{\pi} = \{\pi_N\}_{N\in\N}$ where $\pi_N = \delta_{N-1}$ for all $N\ge 1$, and $\pi_0 =\delta_0$.
Again by Theorem~\ref{thm:Sternfeld-ext-balayage}, the space $\ex{\vdist(\boldsymbol{\pi})}$ is compact and totally disconnected, and we can calculate its extreme points exactly.
For each $a\in\N$, the point $\delta_a^{\boldsymbol{\pi}} = \{\nu_N^a\}_{N\in\N}$ is given by
\begin{equation*}
\nu_N^a = \begin{cases}
\delta_0, &\text{ if } N = 0, \\
\delta_{a-2\lceil\frac{a-N}{2}\rceil}, &\text{ if } 0< N< a, \\
\delta_a, &\text{ if } N\ge a.
\end{cases}
\end{equation*}
for $N\in\N$.
Visually, we can depict $\delta_1^{\boldsymbol{\pi}},\delta_2^{\boldsymbol{\pi}},\delta_3^{\boldsymbol{\pi}}$, and $\delta_4^{\boldsymbol{\pi}}$ as
\begin{equation*}
\begin{bmatrix}
1 & & & \\
1 & 0 & & \\
1 & 0 & 0 & \\
1 & 0 & 0 & 0 \\
\end{bmatrix}, \quad \begin{bmatrix}
0 & & & \\
0 & 1 & & \\
0 & 1 & 0 & \\
0 & 1 & 0 & 0 \\
\end{bmatrix}, \quad \begin{bmatrix}
1 & & & \\
1 & 0 & & \\
0 & 0 & 1 & \\
0 & 0 & 1 & 0 \\
\end{bmatrix}, \text{ and }\begin{bmatrix}
0 & & & \\
0 & 1 & & \\
0 & 1 & 0 & \\
0 & 0 & 0 & 1 \\
\end{bmatrix}.
\end{equation*}
Now we see that $\{\delta_a^{\boldsymbol{\pi}}\}_{a\in\N}$ has two subsequential limits, 
\begin{equation*}
\lim_{a\to\infty}\delta_{2a}^{\boldsymbol{\pi}} = \begin{bmatrix}
0 & & & & \\
0 & 1 & & & \\
0 & 1 & 0 & & \\
0 & 0 & 0 & 1 & \\
\vdots & \vdots & \vdots & \vdots & \ddots \\
\end{bmatrix}, \text{ and } \lim_{a\to\infty}\delta_{2a+1}^{\boldsymbol{\pi}} = \begin{bmatrix}
1 & & & & \\
1 & 0 & & & \\
0 & 0 & 1 & & \\
0 & 0 & 1 & 0 & \\
\vdots & \vdots & \vdots & \vdots & \ddots \\
\end{bmatrix}.
\end{equation*}
Thus, we have
\begin{equation*}
\ex{\vdist(\boldsymbol{\pi})} = \{\delta_a^{\boldsymbol{\pi}}: a\in\N\}\cup\left\{\lim_{a\to\infty}\delta_{2a}^{\boldsymbol{\pi}},\lim_{a\to\infty}\delta_{2a+1}^{\boldsymbol{\pi}}\right\},
\end{equation*}
which is also found in \cite[Example~4.14]{EvansJaffeVMC} with a more hands-on approach.
\end{example}

\begin{example}\label{ex:uniform-balayage}
	Consider the balayage $\boldsymbol{\pi} = \{\pi_N\}_{N\in\N}$ where  $\pi_N(a) = 1/N$ and $\pi_N(0) = 0$ for all $N\in\N$ with $N\ge 1$  and $a\in\llbracket 1,N\rrbracket$.
	In other words, we have $\pi_N = \uniform\llbracket 1,N\rrbracket$ for $N\ge 1$, where $\uniform$ represents the uniform distribution on a finite set.
	We easily compute $\delta_a^{\boldsymbol{\pi}} = \{\nu_N^{a}\}_{N\in\N}$,  where
	\begin{equation*}
	\nu_N^a = \begin{cases}
	\delta_0, &\text{ if } N = 0, \\
	\uniform\llbracket 1,N\rrbracket, &\text{ if } 0< N< a, \\
	\delta_N, &\text{ if } N\ge a,
	\end{cases}
	\end{equation*}
	for $N\in\N$.
	As always, we can visualize $\delta_1^{\boldsymbol{\pi}},\delta_2^{\boldsymbol{\pi}},\delta_3^{\boldsymbol{\pi}}$, and $\delta_4^{\boldsymbol{\pi}}$ as
	\begin{equation*}
	\begin{bmatrix}
	1 & & & \\
	1 & 0 & & \\
	1 & 0 & 0 & \\
	1 & 0 & 0 & 0 \\
	\end{bmatrix}, \quad \begin{bmatrix}
	1 & & & \\
	0 & 1 & & \\
	0 & 1 & 0 & \\
	0 & 1 & 0 & 0 \\
	\end{bmatrix}, \quad \begin{bmatrix}
	1 & & & \\
	\sfrac{1}{2} & \sfrac{1}{2} & & \\
	0 & 0 & 1 & \\
	0 & 0 & 1 & 0 \\
	\end{bmatrix}, \text{ and }\begin{bmatrix}
	1 & & & \\
	\sfrac{1}{2} & \sfrac{1}{2} & & \\
	\sfrac{1}{3} & \sfrac{1}{3} & \sfrac{1}{3} & \\
	0 & 0 & 0 & 1 \\
	\end{bmatrix}.
	\end{equation*}
	Observe $\{\delta_a^{\boldsymbol{\pi}}\}_{a\in\N}$ converges as $a\to\infty$ to $\{\nu_N^{\infty}\}_{N\in\N}$, where
	\begin{equation*}
	\nu_N^{\infty} = \begin{cases}
	\delta_0, &\text{ if } N = 0, \\
	\uniform\llbracket 1,N\rrbracket, &\text{ if } N > 0,
	\end{cases}
	\end{equation*}
	for $N\in\N$, which we depict as
	\begin{equation*}
	\lim_{a\to\infty}\delta_{2a}^{\boldsymbol{\pi}} = \begin{bmatrix}
	1 & & & & \\
	\sfrac{1}{2} & \sfrac{1}{2} & & & \\
	\sfrac{1}{3} & \sfrac{1}{3} & \sfrac{1}{3} & & \\
	\sfrac{1}{4} & \sfrac{1}{4} & \sfrac{1}{4} & \sfrac{1}{4} & \\
	\vdots & \vdots & \vdots & \vdots & \ddots \\
	\end{bmatrix}.
	\end{equation*}
	The question of whether $\{\nu_N^{\infty}\}_{N\in\N}$ is an extreme point of $\vdist(\boldsymbol{\pi})$ was left open in \cite[Example~4.15]{EvansJaffeVMC}, and we now show that the answer is affirmative.
	
	A simple argument by induction shows that the extended balayage probabilities of $\boldsymbol{\pi}$ are given by $\pi_{a,N} = \pi_N$ for $a,N\in\N$ with $a\ge N\ge 1$.
	Now fix $M,N,\in\N$ with $M\le N\le a$, and take $c\in \llbracket 0,M\rrbracket$.
	We easily compute:
	\begin{equation*}
	\begin{split}
	\sum_{b = 0}^{N} (\pi_{b,M}(c))^2\pi_{a,N}(b) &= \sum_{b = 0}^{M} (\delta_{b}(c))^2\pi_{N}(b) +  \sum_{b = M+1}^{N} (\pi_{M}(c))^2\pi_{N}(b) \\
	&= \frac{1}{N}+ \frac{N-M}{NM^2} \\
	&\to \frac{1}{M^2} \\
	&= (\pi_{a,M}(c))^2,
	\end{split}
	\end{equation*}
	uniformly over $a\in\N$.
	In particular, we conclude by Theorem~\ref{thm:Sternfeld-ext-cpt} that $\ex{\vdist(\boldsymbol{\pi})}$ is compact.
	In fact, since we have $\pi_{a,N}(N) = \frac{1}{N}$ for $a\ge N$, it follows that $\{\pi_{a,N}(N)\}_{a\in\N}$ converges for each $N\in\N$, hence by Theorem~\ref{thm:Sternfeld-ext-K0} that $\ex{\vdist(\boldsymbol{\pi})}\setminus\{\delta_a^{\boldsymbol{\pi}}\}_{a\in\N}$ contains a single point.
	By the calculations just performed, this implies
	\begin{equation*}
	\ex{\vdist(\boldsymbol{\pi})} = \{\delta_a^{\boldsymbol{\pi}}: a\in\N\}\cup\left\{\lim_{a\to\infty}\delta_{a}^{\boldsymbol{\pi}}\right\}.
	\end{equation*}
	We also write $\vuniform := \lim_{a\to\infty}\delta_a^{\boldsymbol{\pi}}$, and we call this the \textit{virtual uniform distribution}.
	
	It is instructive to focus a bit more on $\vuniform\in \vdist(\boldsymbol{\pi})$, and to identify its corresponding SMC $\mu := \Phi^{-1}(\vuniform)$.
	By Proposition~\ref{prop:SMC-to-marginals}, it satisfies $\mu(S_0 = 0)  = \mu(S_1 = 1) = 1$, and
	\begin{equation*}
	\mu(S_{N+1} \in \cdot \ | \ S_N) =\frac{N}{N+1}\delta_{S_N}(\cdot) + \frac{1}{N+1}\delta_{N+1}(\cdot)
	\end{equation*}
	for all $N\ge 1$.
	In particular, if we define $A_N = \{S_{N+1} = N+1\}$ for $N\in\N$, then the events $\{A_N\}_{N\in\N}$ are independent and satisfy $\mu(A_N) = 1/(N+1)$.
	Since $\sum_{N\in\N}\mu(A_N) = \infty$ we have by the Borel-Cantelli lemma that $\mu(S_N\to \infty) = 1$.
	Nonetheless, the tail $\sigma$-algebra $\tail = \bigcap_{N\in\N}\sigma(S_N,S_{N+1},\ldots)$ is trivial under $\mu$ by Corollary~\ref{cor:SMC-tail-ext} since $\vuniform\in \ex{\vdist(\boldsymbol{\pi})}$.
\end{example}

\section{Decomposition Theorem}\label{sec:decomp}

In this section we show that any VMC can be decomposed into a countably infinite collection of independent SMCs.
Moreover, these constituent SMCs have a simple description in terms of the VID and VTM representing the given VMC.

For the remainder of this section, we let $(\Omega,\F,\P)$ denote a probability space on which is defined a VMC $\boldsymbol{X} = \{X_N\}_{N\in\N}$.
While in the previous section we viewed staircase processes primarily as probability measures on  $\mathbfcal{S}$, we now view staircase processes as $\mathbfcal{S}$-valued random variables defined on $(\Omega,\F,\P)$.

To begin, fix any $a\in\N$.
For $N\ge a$, define $V_N^a := \#\{k\in\N: I_{X_N,\{a\}}(k) < \infty \}$ to be the number of visits that $X_N$ makes to state $a$.
By the projectivity of $\boldsymbol{X}$, this value does not depend on $N\ge a$ almost surely, hence we can set $V^a$ to be its common value.
(To be concrete, one can take $V^a := V^a_a$.)
By the projectivity of $\boldsymbol{X}$ we also have
\begin{equation*}
\begin{split}
\{V^a > k\} &= \bigcap_{N\ge a}\{I_{X_N,\{a\}}(k) < \infty \}, \text{ and} \\
\{V^a \le k\} &= \bigcap_{N\ge a}\{I_{X_N,\{a\}}(k) = \infty \} \\
\end{split}
\end{equation*}
almost surely, for all $k\in\N$.
Notice that $V^a$ is the number of visits that $\boldsymbol{X}$ makes to $a$, at all levels for which this notion is well-defined.
These considerations leads us to an important classification of the points of $\N$ with respect to the given VMC:

\begin{definition}
With respect to a given VMC, a point $a\in\N$ is called \textit{infinitely-visited} if $\P(V^a = \infty) = 1$, \textit{once-visited} if $\P(V^a = 1) = 1$, \textit{never-visited} if $\P(V^a = 0) = 1$, and \textit{randomly-visited} otherwise.
\end{definition}

\begin{lemma}\label{lem:trasient-prob}
A point $a\in\N$ is randomly-visited if and only if $V^a$ is non-degenerate.
\end{lemma}

\begin{proof}
It is clear that $V^a$ being non-degenerate implies that $a$ cannot be infinitely-visited, once-visited, or never-visited, hence it must be randomly-visited.
For the converse, observe by the strong Markov property of $X_a$, there exist $q_a,p_a\in [0,1]$ such that we have $\P(V^a> k) = q_ap_a^{k}$ for all $k\ge 1$, as well as $\P(V^a> 0) = q_a$.
In this language, observe that $a$ is infinitely-visited if and only if $p_a = q_a = 1$, once-visited if and only if $q_a = 1$ and $p_a= 0$, and never-visited if and only if $q_a =  0$.
By a simple calculation in propositional logic, we see that $a$ being randomly-visited implies that either $q_a \in (0,1)$ or we have $q_a = 1$ and $p_a \in (0,1)$.
In the first case we have $\P(V^a > 0) = q_a \in (0,1)$, and in the second we have $\P(V^a> 1) = p_a \in (0,1)$, hence in either case $V^a$ is non-degenerate.
\end{proof}

Next, take any $N\in\N$ and write $S^0_N := X_N(0)$.
Now for $N,a,k\in\N$ with $a\le N$, we set
\begin{equation*}
S^{a,k}_N := \begin{cases}
X_N(I_{X_N,\{a\}}(k)+1), &\text{ if } V^a> k, \\
0, &\text{ if } V^a \le k.
\end{cases}
\end{equation*}
In words, $\boldsymbol{S}^0 := \{S^0_N\}_{N\in\N}$ consists of all the initial positions of $\boldsymbol{X}$, and $\boldsymbol{S}^{a,k} := \{S^{a,k}_N\}_{N\in\N}$ consists of all of the states occurring immediately after the $k$th visit that $\boldsymbol{X}$ makes to state $a$, at all levels for which such a notion is well-defined.

We are now ready to prove Proposition~\ref{prop:decomp-SMC}, Proposition~\ref{prop:decomp-independence}, and Proposition~\ref{prop:decomp-sigma-alg}, which together constitute the decomposition theorem that was informally stated as Theorem~\ref{thm:intro-decomp} in the introduction.

Note in the following result that a process which is claimed to be a SMC is actually undefined for its first few levels; any uneasiness caused by this slight abuse of terminology will be resolved shortly in Remark~\ref{rem:decomp-extension}.

\begin{proposition}\label{prop:decomp-SMC}
	The process $\boldsymbol{S}^0$ and the processes $\boldsymbol{S}^{a,k}$ for $a,k\in\N$ are SMCs.
\end{proposition}

\begin{proof}
First we show that $\boldsymbol{S}^0$ is a SMC.
To see that it is a staircase process, simply take $N\in\N$ and note that on $\{S_{N+1}^0 < N+1\}= \{X_{N+1}(0) < N+1\}$ we have $I_{X_{N+1},\llbracket 0,N\rrbracket}(0) = 0$, hence $\{X_{N+1}(0) = X_N(0) \} = \{S_{N+1}^0 = S_N^0 \}$.
By intersecting over all $N\in\N$, we conclude $\boldsymbol{S}^0\in \mathbfcal{S}$ almost surely.
To see that it is a MC, take any $M\ge N$ and any $a_{N},\ldots a_M$ such that $\{S_{N+1}^{0} = a_{N+1},\ldots S_{M}^{0} = a_{M}\}$ has positive probability.
If any of $a_{N+1},\ldots a_M$ are equal to 0, then we have $a_{N+1} = 0$ by the projectivity of $\boldsymbol{S}^0$.
In particular, we get
\begin{equation*}
\P(S_N^{0} = 0 \ | \ S_{N+1}^{0} = a_{N+1},\ldots S_{M}^{0} = a_{M}) = \P(S_N^{0} = 0 \ | \ S_{N+1}^{0} = 0)  = 1.
\end{equation*}
Otherwise, all of $a_{N+1},\ldots a_M$ are non-zero, and we set
\begin{equation*}
T^0_L := I_{X_M,\llbracket 0,L\rrbracket}(0)
\end{equation*}
for $L\in \llbracket N,M\rrbracket$.
Note that $\{T^0_L\}_{L= N}^{M}$ are stopping times adapted to the natural filtration of $X_M$, and that almost surely $\{T^0_L\}_{L= N}^{M}$ is non-increasing and $\{T^0_L\}_{L=N+1}^{M}$ are finite.
In particular, we use the strong Markov property of $X_M$ for the following computation:
If $a_N = 0$, then we have
\begin{equation*}
\begin{split}
&\P(S_N^{0} = 0 \ | \ S_{N+1}^{0} = a_{N+1},\ldots S_{M}^{0} = a_{M}) \\
&= \P(T^0_N =\infty \ | \ X_{M}(T^0_{N+1}) = a_{N+1},\ldots X_{M}(T^{0}_M) = a_{M}) \\
&= \P(T^0_N =\infty \ | \ X_{M}(T^0_{N+1}) = a_{N+1}) \\
&= \P(S_N^{0} = 0 \ | \ S_{N+1}^{0} = a_{N+1}),
\end{split}
\end{equation*}
and, if $a_N\neq0$, then we have
\begin{equation*}
\begin{split}
&\P(S_N^{0} = a_N \ | \ S_{N+1}^{0} = a_{N+1},\ldots S_{M}^{0} = a_{M}) \\
&= \P(T^0_N < \infty,X_M(T^0_N) = a_N \ | \ X_{M}(T^0_{N+1}) = a_{N+1},\ldots X_{M}(T^{0}_M) = a_{M}) \\
&= \P(T^0_N < \infty,X_M(T^0_N) = a_N \ | \ X_{M}(T^0_{N+1}) = a_{N+1}) \\
&= \P(S_N^{0} = a_N \ | \ S_{N+1}^{0} = a_{N+1}) \\
\end{split}
\end{equation*}
as claimed.

Next, we fix $a,k\in\N$ and we show that $\boldsymbol{S}^{a,k}$ is a SMC; the proof is essentially the same as the above.
To see that it is a staircase process, consider the event $\{S_{N+1}^{a,k} < N+1\}$ for $N\ge a$.
On $\{V^a> k\}$ we have $\{I_{X_{N+1},\{a\}}(k) < \infty \}$ almost surely, hence
\begin{equation*}
\begin{split}
\{S_{N+1}^{a,k} < N+1 \} &\subseteq \{X_{N+1}(I_{X_{N+1},\{a\}}(k) +1) < N+1 \} \\
&\subseteq \{X_{N+1}(I_{X_{N+1},\{a\}}(k) +1) = X_{N}(I_{X_{N},\{a\}}(k) +1) \} \\
&\subseteq \{S_{N+1}^{a,k} = S_{N}^{a,k} \},
\end{split}
\end{equation*}
almost surely, and on $\{V^a\le k\}$ we have $\{I_{X_{N+1},\{a\}}(k) = \infty \}$ almost surely, hence $\{S_{N+1}^{a,k} < N+1\} \subseteq \{S_{N+1}^{a,k} = S_N^{a,k} =  0\}$ almost surely.
Thus, we have $\{S_{N+1}^{a,k} < N+1\}\subseteq \{S_{N+1}^{a,k} = S_N^{a,k}\}$ almost surely for all $N\in\N$, and taking intersections gives $\boldsymbol{S}^{a,k}\in \mathbfcal{S}$ almost surely.

Next let us show that $\boldsymbol{S}^{a,k}$ is a MC.
Indeed, take any $M\ge N$, and let $a_{N},\ldots a_M\in\N$ be such that $\{S_{N+1}^{a,k} = a_{N+1},\ldots S_{M}^{a,k}  = a_M \}$ has positive probability.
If any of $a_{N+1},\ldots a_M$ are equal to 0, then we have $a_{N+1} = 0$  by the projectivity of $\boldsymbol{S}^{a,k}$, hence
\begin{equation*}
\P(S_N^{a,k} = a_N \ | \ S_{N+1}^{a,k} = a_{N+1},\ldots S_{M}^{a,k} = a_{M}) = \P(S_N^{a,k} = a_N \ | \ S_{N+1}^{a,k} = a_{N+1}) = 1.
\end{equation*}
Otherwise, all $a_{N+1},\ldots a_M$ are non-zero.
In this case, define
\begin{equation*}
T_L^{a,k} := I_{X_M,\llbracket 0,L\rrbracket}(I_{X_{L},\{a\}}(k)+1)
\end{equation*}
for $L\in \llbracket N,M\rrbracket$, and note that $\{T_L^{a,k}\}_{L=N}^{M}$ is almost surely non-increasing and $\{T_L^{a,k}\}_{L=N+1}^{M}$ are almost surely finite.
Additionally, $\{T_L^{a,k}\}_{L=N}^{M}$ are all stopping times with respect to the natural filtration of $X_M$.
Thus, we compute by the  strong Markov property of $X_M$:
If $a_N = 0$, then
\begin{equation*}
\begin{split}
&\P(S_N^{a,k}= 0 \ | S_{N+1}^{a,k} = a_{N+1},\ldots S_M^{a,k} = a_M) \\
&= \P(T_N^{a,k} = \infty \ | \ X_M(T^{a,k}_{N+1}) = a_{N+1},\ldots X_M(T^{a,k}_M) =a_M) \\
&= \P(T^{a,k}_N = \infty \ | \ X_M(T^{a,k}_{N+1}) = a_{N+1}) \\
&= \P(S_N^{a,k}= 0 \ | S_{N+1}^{a,k} = a_{N+1}),
\end{split}
\end{equation*}
and if $a_N \neq 0$, then
\begin{equation*}
\begin{split}
&\P(S_N^{a,k}= a_N \ |\ S_{N+1}^{a,k} = a_{N+1},\ldots S_M^{a,k} = a_M) \\
&= \P(T^{a,k}_N < \infty,X_M(T^{a,k}_N) = a_N \ | \ X_M(T^{a,k}_{N+1}) = a_{N+1},\ldots X_M(T^{a,k}_M) =a_M) \\
&= \P(T^{a,k}_N < \infty,X_M(T^{a,k}_N) = a_N \ | \ X_M(T^{a,k}_{N+1}) = a_{N+1}) \\
&= \P(S_N^{a,k}= a_N \ | S_{N+1}^{a,k} = a_{N+1}).
\end{split}
\end{equation*}
This shows that $\boldsymbol{S}^{a,k}$ is a MC hence a SMC, and this finishes the proof.
\end{proof}

\begin{proposition}\label{prop:decomp-independence}
	The processes $\{\boldsymbol{S}^0\}\cup\{\boldsymbol{S}^{a,k}\}_{a,k\in\N}$ are mutually independent.
\end{proposition}

\begin{proof}
	Take any $n\in\N$ and any $a_1,\ldots a_{n},k_1,\ldots k_{n},N_0,\ldots N_n\in\N$ such that we have $(a_i,k_i)\neq (a_j,k_j)$ for all distinct $i,j\in \llbracket 0,n\rrbracket$.
	Write $N := \max\{N_0,\ldots N_n\}$.
	By the strong Markov property of $X_N$, the random variables
	\begin{equation*}
	\{S^0_{N_0},S^{a_1,k_1}_{N_1},\ldots S^{a_n,k_n}_{N_n}\}
	\end{equation*}
	are mutually independent.
	Therefore we conclude by the following result, whose proof is a simple exercise in measure theory and is thus omitted.
\end{proof}

\begin{lemma}
	Let $\{Z_{\ell}\}_{\ell\in\N}$ be a collection of discrete-time stochastic processes, and write $Z_{\ell} = \{Z_{\ell}(n)\}_{n\in\N}$ for all $\ell\in\N$.
	Then, the following are equivalent:
	\begin{enumerate}
		\item[(i)] $\{Z_{\ell}\}_{\ell\in\N}$ are independent, and
		\item[(ii)] $\{Z_0(n_0),\ldots Z_L(n_L) \}$ are independent for any $L\in\N$ and $n_0,\ldots n_{L}\in\N$.
	\end{enumerate}
\end{lemma}

\begin{proposition}\label{prop:decomp-sigma-alg}
	We have $\sigma(\boldsymbol{X}) = \sigma(\boldsymbol{S}^0) \vee \sigma(\boldsymbol{S}^{a,k}: a,k\in \N)$.
\end{proposition}

\begin{proof}
Observe by the projectivity property of $\vchain$ that $\sigma(\boldsymbol{X})$ is generated by sets of the form
\begin{equation}\label{eqn:VMC-generators}
\{X_N(0)= a_0,\ldots X_N(\ell) = a_{\ell} \}
\end{equation}
for $N\in\N$ and $\ell\in\N$ and $a_0,\ldots a_{\ell}\in\llbracket 0,N\rrbracket$ arbitrary.
Thus, it suffices to show that all the sets \eqref{eqn:VMC-generators} are in $\sigma(\boldsymbol{S}^0) \vee \sigma(\boldsymbol{S}^{a,k}: a,k\in \N)$.
We prove this for fixed $N\in\N$ by induction on $\ell\in\N$.
The base case of $\ell=0$ is true since $X_N(0) = S_N^0$ is $\sigma(\boldsymbol{S}^0)$-measurable.
For the inductive step suppose the claim holds for $\ell\in\N$, and take arbitrary $a_0,\ldots a_{\ell+1}\in \llbracket 0,N\rrbracket$.
Set $k =\#\{i\in \llbracket0,\ell\rrbracket:  a_i = a_{\ell} \}$ and note that we have
\begin{equation*}
\begin{split}
\{X_N(0) &= a_0,\ldots X_N(\ell+1)  = a_{\ell+1}\} \\
&= \{X_N(0) = a_0,\ldots X_N(\ell)  = a_{\ell}\} \cap \left\{S_N^{a_{\ell},k-1} = a_{\ell+1}\right\}.
\end{split}
\end{equation*}
The right side lies $\sigma(\boldsymbol{S}^0) \vee \sigma(\boldsymbol{S}^{a,k}: a,k\in \N)$ by the inductive hypothesis, hence the induction is complete and we have $\sigma(\boldsymbol{X}) \subseteq \sigma(\boldsymbol{S}^0) \vee \sigma(\boldsymbol{S}^{a,k}: a,k\in \N)$.
The opposite inclusion $\sigma(\boldsymbol{X}) \supseteq \sigma(\boldsymbol{S}^0) \vee \sigma(\boldsymbol{S}^{a,k}: a,k\in \N)$ is trivial, so the proof is complete.
\end{proof}

These results together constitute the rigorous statement and proof of the informal Theorem~\ref{thm:intro-decomp} given in the introduction.
In words, it states that every VMC $\boldsymbol{X}$ gives rise to an independent collection of SMCs $\{\boldsymbol{S}^0\}\cup\{\boldsymbol{S}^{a,k}: a,k\in\N\}$ which generates the same $\sigma$-algebra.
We call $\{\boldsymbol{S}^0\}\cup\{\boldsymbol{S}^{a,k}: a,k\in\N\}$ the \textit{staircase decomposition} of $\boldsymbol{X}$.

It may be instructive to see an example of this correspondence applied to a single sample path of a VMC.

\begin{example}
Suppose that a VMC $\boldsymbol{X} = \{X_N\}_{N\in\N}$ is realized so that the initial segments of the sample paths of its first few levels are:
\begin{equation*}
\begin{matrix}
X_5: & 4 & 5 & 2 & 3 & 1 & \cdots \\
X_4: & 4 & 2 & 3 & 1 & 4 & \cdots \\
X_3: & 2 & 3 & 1 & 1 & 2 & \cdots \\
X_2: & 2 & 1 & 1 & 2 & 0 & \cdots \\
X_1: & 1 & 1 & 0 & 0  & 0 & \cdots \\
X_0: & 0 & 0 & 0 & 0  & 0 & \cdots \\
\end{matrix}
\end{equation*}
Then the first few constituents of its straircase decomposition are:
\begin{equation*}
\begin{matrix}
\boldsymbol{S}^0: & 0 & 1 & 2 & 2 & 4 & 4 & \cdots \\
\end{matrix}
\end{equation*}
as well as
\begin{equation*}
\begin{matrix}
\boldsymbol{S}^{1,0}: & \, & 1 & 1 & 1 & 4 & 5 & \cdots \\
\boldsymbol{S}^{1,1}: & \, & 0 & 2 & 2 & 2 & 2 & \cdots \\
\end{matrix}
\end{equation*}
and
\begin{equation*}
\begin{matrix}
\boldsymbol{S}^{2,0}: & \, & \, & 1 & 3 & 3 & 3 & \cdots \\
\boldsymbol{S}^{2,1}: & \, & \, & 0 & 0 & 0 & 0 & \cdots \\
\end{matrix}
\end{equation*}
Recall here that $\boldsymbol{S}_N^{a,k}$ is only defined for $N\ge a$, although note Remark~\ref{rem:decomp-extension} in which it will be  shown that one can canonically extend the process to all levels $N\in\N$.
\end{example}

Since $\boldsymbol{S}^0$ and $\boldsymbol{S}^{a,k}$ for $a,k\in\N$ are SMCs, we are led by the previous section to ask whether the elements of $\vdist$ given by $\Phi(\P\circ (\boldsymbol{S}^0)^{-1})$ and $\Phi(\P\circ (\boldsymbol{S}^{a,k})^{-1})$ for $a,k\in\N$ can be easily understood.
Indeed, as we show next, these depend very directly on the VID and VTM representing $\boldsymbol{X}$.

For a moment we digress to discuss a structure embedded within VTMs.
For any VTM $\boldsymbol{K} =  \{K_N\}_{N\in\N}$ define the balayage $\boldsymbol{\pi}^{\boldsymbol{K}} =\{\pi^{\boldsymbol{K}}_N\}_{N\in\N}$ via \eqref{eqn:VTM-balayage}.
For convenience, we replace $\boldsymbol{\pi}^{\boldsymbol{K}}$ with $\boldsymbol{K}$ in most expressions where there is no risk of amgibuity, for example
\begin{equation*}
\vdist(\boldsymbol{K}):= \vdist(\boldsymbol{\pi}^{\boldsymbol{K}}), \qquad \mathcal{P}_{\mathrm{MC}}^{\boldsymbol{K}}(\mathbfcal{S}) := \mathcal{P}_{\mathrm{MC}}^{\boldsymbol{\pi}^{\boldsymbol{K}}}(\mathbfcal{S}), \qquad \text{ and } \qquad \delta_{a}^{\boldsymbol{K}} := \delta_{a}^{\boldsymbol{\pi}^{\boldsymbol{K}}}
\end{equation*}
for $a\in\N$.
Note in this case that $\pi_N^{\boldsymbol{K}}\in\mathcal{P}(\llbracket 0,N\rrbracket)$ for $N\in\N$ is exactly the balayage distribution of a particle started at $N+1$ moving according to the transition matrix $K_{N+1}$ and stopped when it first hits $\llbracket 0,N\rrbracket$, along with the convention that a particle forever trapped at $N+1$ gets sent to 0.
Now note by Lemma~\ref{lem:SMC-marginals-residual} that for each VTM $\boldsymbol{K}$ and each $a\in\N$ there is a unique element of  $\vdist(\boldsymbol{K})$ which is eventually equal to $\{K_N(a,\cdot)\}_{N\ge a}$; this element is denoted $\boldsymbol{K}(a,\cdot)$ and can be interpreted as the ``$a$th row'' of $\boldsymbol{K}$.
In analogy with the case of classical MCs, observe that if $\boldsymbol{X} = \{X_N\}_{N\in\N}$ is a VTM represented by $(\delta_a^{\boldsymbol{K}},\boldsymbol{K})$, then we have $\P\circ (\{X_N(1)\}_{N\in\N})^{-1} = \Phi^{-1}(\boldsymbol{K}(a,\cdot))$.

Now back to the task at hand.
Recall that $(\Omega,\F,\P)$ is a probability space on which is defined a VMC $\boldsymbol{X} = \{X_N\}_{N\in\N}$, and that $\{\boldsymbol{S}^0\}\cup\{\boldsymbol{S}^{a,k}: a,k\in\N\}$ is the staircase decomposition of $\boldsymbol{X}$.
Also assume that $\boldsymbol{X}$ is represented by the VID $\boldsymbol{\nu}$ and VTM $\boldsymbol{K}$.
Then we have the following.

\begin{lemma}\label{lem:SMC-dist}
We have $\Phi(\P\circ(\boldsymbol{S}^{0})^{-1})= \boldsymbol{\nu}$.
Moreover, for each $a,k\in\N$ we have
\begin{equation*}
\begin{split}
\Phi\left(\P(\ \cdot \ | \ V^a> k)\circ(\boldsymbol{S}^{a,k})^{-1}\right) &= \boldsymbol{K}(a,\cdot) \\
\Phi\left(\P(\ \cdot \ | \ V^a\le k)\circ(\boldsymbol{S}^{a,k})^{-1}\right) &= \boldsymbol{0},
\end{split}
\end{equation*}
hence
\begin{equation*}
\Phi\left(\P\circ(\boldsymbol{S}^{a,k})^{-1}\right) = \P(V^a> k)\boldsymbol{K}(a,\cdot) + \P(V^a \le k)\boldsymbol{0}.
\end{equation*}
\end{lemma}

\begin{proof}
Write $\boldsymbol{\nu} = \{\nu_N\}_{N\in\N}$ and $\boldsymbol{K} =\{K_N\}_{N\in\N}$.
The first claim is trivial since we have $\P(S_N^0 = b) = \P(X_N(0) = b) = \nu_{N}(b)$ by construction for all $N\in\N$ and $b\in \llbracket 0,N\rrbracket$.
For the second claim, take $a,k\in\N$.
Note that we have $\{V^a> k\} = \bigcap_{N\ge a}\{I_{X_N,\{a\}}(k) < \infty\}$ almost surely.
Thus for any $N\ge a$ and $b\in \llbracket 0,N\rrbracket$ we use the strong Markov property of $X_N$ to compute
\begin{equation*}
\begin{split}
\P(S_N^{a,k} = b \ | \  V^a> k) &= \P(X_N(I_{X_N,\{a\}}(k)+1) = b \ | \  V^a> k) \\
&= \P(X_N(1) = b \ | \  X_N(0) = a) \\
&= K_N(a,b).
\end{split}
\end{equation*}
Moreover, we have $\{V^a\le k\} = \bigcap_{N\ge a}\{I_{X_N,\{a\}}(k) = \infty\}$ almost surely, so for any $N\ge a$ and $b\in \llbracket 0,N\rrbracket$ we have $\P(S_N^{a,k} = 0 \ | \  V^a\le k) = 1$ by construction.
\end{proof}

\begin{remark}\label{rem:decomp-extension}
For $a,k\in\N$, the SMC $\boldsymbol{S}^{a,k} = \{S_N^{a,k}\}_{N\ge a}$ is, a priori, only defined for levels $N \ge a$.
However, by Lemma~\ref{lem:SMC-marginals-residual} and Lemma~\ref{lem:SMC-dist}, there is a unique sequence of marginals in $\vdist(\boldsymbol{K})$ which eventually agrees with the sequence of marginals of $\boldsymbol{S}^{a,k}$.
Thus it well-defined to extend $\{S_N^{a,k}\}_{N\ge a}$ to $\{S_N^{a,k}\}_{N\in\N}$, at least in the sense of the distribution of this process.
Since the processes $\{\boldsymbol{S}^0\}\cup\{\boldsymbol{S}^{a,k}: a,k\in\N\}$ are independent by Proposition~\ref{prop:decomp-independence}, we might as well perform this extension independently across all $a,k\in\N$.
Therefore, it is no loss of generality to assume that $\boldsymbol{S}^{a,k}$ is a bona fide SMC with the uniquely determined sequence of marginals.
\end{remark}

To conclude this section, let us see some concrete examples of the staircase decomposition of some VMCs.
Recall our convention that square brackets are used to represent that the 0th row and column have been omitted from a matrix.
In the case of VTMs this causes no confusion since the 0th row is fully determined by the state 0 being absorbing, and since the 0th column can always be determined from the remaining columns.

\begin{example}\label{ex:down-from-infty-1}
Let $\{q_N\}_{N\in\N}$ be any sequence in $[0,1]$, and define the VTM $\boldsymbol{K} = \{K_N\}_{N\in\N}$ for $N\in\N$ via
\begin{equation*}
K_N = \begin{bmatrix}
0 & 0 & 0 & 0 & \cdots & 0 & 0 & 1 \\
q_1 & 0 & 0 & 0 & \cdots & 0 & 0 & 1-q_1 \\
0 & q_2 & 0 & 0 & \cdots & 0 & 0 & 1-q_2 \\
0 & 0 & q_3& 0 & \cdots & 0 & 0 & 1-q_3 \\
\vdots & \vdots & \vdots & \vdots & \ddots & \vdots & \vdots & \vdots \\
0 & 0 & 0 & 0 & \cdots & 0 & 0& 1-q_{N-3} \\
0 & 0 & 0 & 0 & \cdots & q_{N-2} & 0 & 1-q_{N-2} \\
0 & 0 & 0 & 0 & \cdots & 0 & q_{N-1} & 1-q_{N-1} \\
\end{bmatrix}.
\end{equation*}
In \cite[Example~2.23]{EvansJaffeVMC} it is observed that this VTM, heuristically speaking, corresponds to a VMC which tries to ``come down from infinity'' but which ``jumps back to infinity'' at some random times.
We now observe that the staircase decomposition provides a way to make this idea rigorous.
Indeed, observe that $\boldsymbol{\pi}^{\boldsymbol{K}}$ is exactly the balayage defined in Example~\ref{ex:down-balayage}, that $\lim_{a\to\infty}\delta_a^{\boldsymbol{K}}$ corresponds to the unique ``state at infinity'' by Remark~\ref{rem:states-compactification}, and that we have
\begin{equation*}
\begin{split}
\boldsymbol{K}(1,\cdot) &= \lim_{a\to\infty}\delta_a^{\boldsymbol{K}}, \text{ and } \\
\boldsymbol{K}(a,\cdot) &= q_a\delta_{a-1}^{\boldsymbol{K}}+(1-q_a)\left(\lim_{a\to\infty}\delta_a^{\boldsymbol{K}}\right), \text{ for } a\ge 2.
\end{split}
\end{equation*}
\end{example}

\begin{example}\label{ex:down-from-infty-2}
Consider the VTM  $\boldsymbol{K} = \{K_N\}_{N\in\N}$ defined for $N\in\N$ via
\begin{equation*}
K_N = \begin{bmatrix}
0 & 0 & 0 & 0 & \cdots & 0 & 0 & 1/2 & 1/2 \\
0 & 0 & 0 & 0 & \cdots & 0 & 0 & 1/2 & 1/2 \\
1 & 0 & 0 & 0 & \cdots & 0 & 0 & 0 & 0 \\
0 & 1 & 0 & 0 & \cdots & 0 & 0 & 0 & 0 \\
\vdots & \vdots & \vdots & \vdots & \ddots & \vdots & \vdots & \vdots & \vdots \\
0 & 0 & 0 & 0 & \cdots & 0 & 0 & 0 & 0 \\
0 & 0 & 0 & 0 & \cdots & 0 & 0 & 0 & 0 \\
0 & 0 & 0 & 0 & \cdots & 1 & 0 & 0 & 0 \\
0 & 0 & 0 & 0 & \cdots & 0 & 1 & 0 & 0 \\
\end{bmatrix}.
\end{equation*}
It is observed in \cite[Example~2.24]{EvansJaffeVMC} that this VTM roughly corresponds to a VMC that has ``two ways to come down from infinity'' which it chooses uniformly at random during each excursion.
As before, we can make this idea rigorous by inspecting its staircase decomposition.
Note that $\boldsymbol{\pi}^{\boldsymbol{K}}$ is the balayage from Example~\ref{ex:two-down-balayage}, and that we have
\begin{equation*}
\begin{split}
\boldsymbol{K}(1,\cdot) = \boldsymbol{K}(2,\cdot) &= \frac{1}{2}\left(\lim_{a\to\infty}\delta_{2a}^{\boldsymbol{K}}\right)+\frac{1}{2}\left(\lim_{a\to\infty}\delta_{2a+1}^{\boldsymbol{K}}\right), \text{ and } \\
\boldsymbol{K}(a,\cdot) &= \delta_{a-2}^{\boldsymbol{K}}, \text{ for } a\ge 3.
\end{split}
\end{equation*}
\end{example}

\begin{example}\label{ex:inf-clique}
Define the VTM $\mathbf{K} = \{K_N\}_{N\in\N}$ via
\begin{equation*}
K_N = \begin{bmatrix}
1/N & 1/N & \cdots & 1/N & 1/N \\
1/N & 1/N & \cdots & 1/N & 1/N \\
\vdots & \vdots & \ddots & \vdots & \vdots \\
1/N & 1/N & \cdots & 1/N & 1/N \\
1/N & 1/N & \cdots & 1/N & 1/N \\
\end{bmatrix}
\end{equation*}
for all $N\in\N$.
Notice  that $\boldsymbol{\pi}^{\boldsymbol{K}}$ is the balayage from Example~\ref{ex:uniform-balayage}, and that we have $\boldsymbol{K}(a,\cdot) = \vuniform\in \ex{\vdist(\boldsymbol{K})}$ for all $a\ge 1$.
This also rigorously establishes the heuristic idea from \cite[Example~4.15]{EvansJaffeVMC} that $\boldsymbol{K}$ represents the ``random walk on the infinite clique''.
\end{example}

\section{The Zero-One Law}\label{sec:tail}

Finally we come to the main application of the theory developed in this paper, that is, the characterization of the triviality of the tail $\sigma$-algebra of virtual Markov chains (VMCs).
To do this, let $(\Omega,\F,\P)$ be a probability space on which is defined a VMC $\boldsymbol{X} = \{X_N\}_{N\in\N}$.
By the tail $\sigma$-algebra of $\boldsymbol{X}$ we mean
\begin{equation*}
\tail(\boldsymbol{X}) := \bigcap_{N\in\N}\sigma(X_N,X_{N+1},\ldots).
\end{equation*}
First, let us show that the question is uninteresting for classical VMCs.

\begin{proposition}\label{prop:CMC-trivial}
	If a VMC $\boldsymbol{X}$ is classical, then $\tail(\boldsymbol{X}) = \sigma(\boldsymbol{X})$ almost surely.
\end{proposition}

\begin{proof}
We of course have $\tail(\boldsymbol{X})\subseteq \sigma(\boldsymbol{X})$, so we only need to show $\sigma(\boldsymbol{X})\subseteq \tail(\boldsymbol{X})$ almost surely.
Indeed, recall that $\boldsymbol{X} = \{X_N\}_{N\in\N}$ being classical implies that almost surely $X:= \lim_{N\to\infty}X_N$ exists in $\chain$ and satisfies $X_N  = P_N(X)$ for all $N\in\N$.
In particular, $X$ is almost surely $\tail(\boldsymbol{X})$-measurable, so each $X_N$ is almost surely $\tail(\boldsymbol{X})$-measurable.
Therefore, $\tail(\boldsymbol{X}) = \sigma(\boldsymbol{X})$ almost surely.
\end{proof}

Now we consider the general case in which the question becomes interesting.
For the remainder of the section, suppose that $\boldsymbol{X}$ is represented by $(\boldsymbol{\nu},\boldsymbol{K})$.
Let $\{\boldsymbol{S}^{0}\}\cup\{\boldsymbol{S}^{a,k}: a,k\in\N \}$ denote the staircase decomposition of $\boldsymbol{X}$ as defined in Section~\ref{sec:decomp}, where $\boldsymbol{S}^0 = \{S_N^0\}_{N\in\N}$ and $\boldsymbol{S}^{a,k} = \{S_N^{a,k}\}_{N\in\N}$ for $a,k\in\N$.
Write
\begin{equation*}
\tail^0 := \bigcap_{N\in\N}\sigma(S^0_N,S^0_{N+1},\ldots)
\end{equation*}
for the tail $\sigma$-algebra of $\boldsymbol{S}^0$, and
\begin{equation*}
\tail^{a,k} := \bigcap_{N\in\N}\sigma(S^{a,k}_N,S^{a,k}_{N+1},\ldots)
\end{equation*}
for the tail $\sigma$-algebra of $\boldsymbol{S}^{a,k}$, for each $a,k\in\N$.
Then, we have the following.

\begin{lemma}\label{lem:VMC-tail-sigma-alg}
	We have $\tail(\boldsymbol{X}) = \tail^0 \vee \sigma(\tail^{a,k}: a,k\in \N)$ almost surely.
\end{lemma}

\begin{proof}
	For each $N\in\N$, set $\F^0_{N} := \sigma(S_N^0,S_{N+1}^0,\ldots)$.
	Similarly, for $N,a,k\in\N$, set $\F^{a,k}_{N} := \sigma(S_N^{a,k},S_{N+1}^{a,k},\ldots)$.
	We do not require $N\ge a$, and it is implicitly assumed that we have independently extended the staircase decomposition to all scales in the  manner of Remark~\ref{rem:decomp-extension}.
	Now take any $N\in\N$ and observe by the projectivity property of $\vchain$ that $\sigma(X_N,X_{N+1},\ldots)$ is generated by sets of the form
	\begin{equation}\label{eqn:tail-generators}
	\{X_M(0)= a_0,\ldots X_M(\ell) = a_{\ell} \}
	\end{equation}
	for $M\ge N$ and $\ell\in\N$ and $a_0,\ldots a_{\ell}\in\llbracket 0,M\rrbracket$.
	By repeating the proof of Proposition~\ref{prop:decomp-sigma-alg}, we find that all sets of the form \eqref{eqn:tail-generators} are in $\F_N^0\vee \sigma(\F_N^{a,k}: a,k\in\N)$, hence we conclude
	\begin{equation}\label{eqn:tail-eqn-1}
	\tail(\boldsymbol{X}) \subseteq\sigma(X_N,X_{N+1},\ldots)\subseteq \F_N^0\vee \sigma(\F_N^{a,k}: a,k\in\N)
	\end{equation}
	for all $N\in\N$.
	
	Next, we use the backwards martingale convergence theorem \cite[Theorem~7.23]{Kallenberg} to see that for all bounded measurable $Z$ we have
	\begin{equation*}
	\E[Z\ |\ \F_N^0] \to \E[Z\ | \  \tail^0]
	\end{equation*}
	and
	\begin{equation*}
	\E[Z\ |\ \F_N^{a,k}] \to \E[Z\ | \  \tail^{a,k}]
	\end{equation*}
	for all $a,k\in\N$ and holding almost surely as $N\to\infty$.
	Combining this with the independence established in Proposition~\ref{prop:decomp-independence} and a standard monotone class argument, we see that for all bounded measurable $Z$ we have
	\begin{equation*}
	\E[Z\ |\ \sigma(\F_N^0)\vee\sigma(\F_N^{a,k}:a,k\in\N)] \to \E[Z\ | \  \sigma(\tail^0)\vee\sigma(\tail^{a,k}:a,k\in\N)]
	\end{equation*}
	holding almost surely as $N\to\infty$.
	In particular, this shows that if $Z$ is bounded and $\bigcap_{N\in\N}(\sigma(\F_N^0)\vee\sigma(\F_N^{a,k}:a,k\in\N))$-measurable, then we have $Z = \E[Z\ | \ \sigma(\tail^0)\vee\sigma(\tail^{a,k}:a,k\in\N)]$  holding almost surely.
	In other words, any such $Z$ is almost surely equal to some $\sigma(\tail^0)\vee\sigma(\tail^{a,k}:a,k\in\N)$-measurable random variable, and this establishes 
	\begin{equation}\label{eqn:tail-eqn-2}
	\bigcap_{N\in\N}(\F_N^0\vee \sigma(\F_N^{a,k}: a,k\in\N)) \subseteq \tail^0\vee \sigma(\tail^{a,k}: a,k\in\N)
	\end{equation}
	almost surely.
	
	Combining \eqref{eqn:tail-eqn-1} and \eqref{eqn:tail-eqn-2} yields $\tail(\boldsymbol{X}) \subseteq  \sigma(\tail^0)\vee\sigma(\tail^{a,k}:a,k\in\N)$ almost surely.
	Conversely, for all $N\in\N$, the random variables $S_N^0$ and $S_N^{a,k}$ for $a,k\in\N$ are $X_N$-measurable, so we clearly have $\tail(\boldsymbol{X}) \supseteq \tail^0 \vee \sigma(\tail^{a,k}: a,k\in \N)$.
	This finishes the proof.
\end{proof}

These preparations finally lead us to the main result of the paper.
The following, is our main zero-one law for $\tail(\boldsymbol{X})$.

\begin{theorem}\label{thm:general-ZOL}
	A VMC represented by $(\boldsymbol{\nu},\boldsymbol{K})$ has a trivial tail $\sigma$-algebra $\tail(\boldsymbol{X})$ if and only if the conditions
	\begin{enumerate}
	\item[(i)] $\boldsymbol{\nu}\in \ex{\vdist(\boldsymbol{K})}$,
	\item[(ii)] $\boldsymbol{K}(a,\cdot)\in \ex{\vdist(\boldsymbol{K})}$ if $a\in\N$ is infinitely-visited or once-visited, and
	\item[(iii)] there are no randomly-visited $a\in\N$,
	\end{enumerate}
	are all satisfied
\end{theorem}

\begin{proof}
We begin by collecting some observations:
By Lemma~\ref{lem:VMC-tail-sigma-alg}, we know that $\tail(\boldsymbol{X})$ is trivial if and only if $\{\tail^0\} \cup\{\tail^{a,k}: a,k\in\N\}$ are all trivial.
Moreover, by Theorem~\ref{thm:SMC-tail-ext}, the triviality of the tail $\sigma$-algebra of a SMC is equivalent to the extremality of its image under $\Phi$.
Finally, by Lemma~\ref{lem:SMC-dist} we have an explicit form of these images.
With these observations in mind, we consider the two directions of the equivalence.

For the first direction, suppose that (i), (ii), and (iii) all hold and let us show the triviality of $\tail(\boldsymbol{X})$.
Note that (i) implies that $\tail^0$ is trivial.
Now take $a\in\N$, and note by (iii) that it suffices to consider three cases:
If $a$ is infinitely-visited, then (ii) implies that $\tail^{a,k}$ is trivial for all $k\in\N$.
If $a$ is once-visited, then $\tail^{a,k}$ is trivial for all $k\ge 1$, and (ii) implies that $\tail^{a,0}$ is trivial.
Finally, if $a$ is never-visited, then $\tail^{a,k}$ trivial for all $k\in\N$.

For the converse direction, let us show that the failure of any of (i), (ii), or (iii) implies the non-triviality of $\tail(\boldsymbol{X})$.
If (i) fails, then $\tail^0$ is non-trivial.
If (ii) fails, then there are two further cases to consider:
in the first case there is a infinitely-visited $a\in\N$ with $\boldsymbol{K}(a,\cdot)\notin \ex{\vdist(\boldsymbol{K})}$, hence it follows that $\tail^{a,k}$ is non-trivial for all $k\in\N$; in the second case there is a once-visited $a\in\N$ with $\boldsymbol{K}(a,\cdot)\notin \ex{\vdist(\boldsymbol{K})}$, hence it follows that $\tail^{a,0}$ is non-trivial.
Finally, if (iii) fails, then some $a\in\N$ is randomly-visited, and it follows  that $V^a$ is a non-trivial $\tail(\boldsymbol{X})$-measurable random variable by Lemma~\ref{lem:trasient-prob}.
\end{proof}

Under a natural condition on the VTM $\boldsymbol{K}$, the preceding theorem can be simplified slightly.
Recall that a VTM $\boldsymbol{K} = \{K_N\}_{N\in\N}$ is called \textit{irreducible} if $\llbracket 1,N\rrbracket$ is an irreducible class for $K_N$ for each $N\in\N$.
Moreover, a VMC is called \textit{irreducible} if it can be represented by a compatible pair $(\boldsymbol{\nu},\boldsymbol{K})$  where $\boldsymbol{K}$ is irreducible.
Then we have the following, which was stated as Theorem~\ref{thm:intro-ZOL} in the introduction.

\begin{theorem}\label{thm:irred-ZOL}
	An irreducible VMC represented by $(\boldsymbol{\nu},\boldsymbol{K})$ has a trivial tail $\sigma$-algebra $\tail(\boldsymbol{X})$ if and only if the conditions
	\begin{enumerate}
		\item[(i)] $\boldsymbol{\nu}\in \ex{\vdist(\boldsymbol{K})}$, and
		\item[(ii)] $\boldsymbol{K}(a,\cdot)\in \ex{\vdist(\boldsymbol{K})}$ for all $a\ge 1$,
	\end{enumerate}
	are both satisfied.
\end{theorem}

\begin{proof}
If $\boldsymbol{K}$ is irreducible, then for each $a\ge 1$, the MC $X_a$ has all states recurrent, hence $a$ is infinitely-visited.
Thus, the result follows from Theorem~\ref{thm:general-ZOL}.
\end{proof}

Let us conclude by considering some examples.
For instance, let us consider Example~\ref{ex:down-from-infty-1} and Example~\ref{ex:down-from-infty-2}.
Both VTMs are irreducible, and the analyses therein imply that (ii)  of Theorem~\ref{thm:irred-ZOL} fails, hence that the tail $\sigma$-algebra $\tail(\boldsymbol{X})$ is non-trivial.
In fact, this is true for any VID $\boldsymbol{\nu}$.
Of course, we did not necessarily need to the zero-one law for this conclusion, since it is also easy to directly construct a non-trivial $\tail$-measurable random variable for these concrete cases.

More interesting is Example~\ref{ex:inf-clique}.
Since the VTM in this example is also irreducible we can apply Theorem~\ref{thm:irred-ZOL}.
The analysis in Example~\ref{ex:inf-clique} shows that we have $\boldsymbol{K}(a,\cdot) = \vuniform\in \ex{\vdist(\boldsymbol{K})}$ for all $a\ge 1$, hence that the tail $\sigma$-algebra $\tail(\boldsymbol{X})$ is trivial if and only if its VID satisfies $\boldsymbol{\nu}\in \ex{\vdist(\boldsymbol{K})}$.
This yields a collection of non-trivial VMCs whose tails are trivial.

\nocite{*}
\bibliography{VMC_II}
\bibliographystyle{plain}

\end{document}